\documentclass[12pt]{article}
\usepackage[centertags]{amsmath}
\usepackage{amsfonts}
\usepackage{amssymb}

\usepackage{amsthm}
\usepackage[top=3cm, bottom=3cm,left=3cm, right=3cm]{geometry}
\usepackage[all]{xy}

\begin{document}
 \newtheorem{thm}{Theorem}[section]
 \newtheorem{cor}[thm]{Corollary}
 \newtheorem{lem}[thm]{Lemma}
 \newtheorem{prop}[thm]{Proposition}
 \theoremstyle{definition}
 \newtheorem{defn}[thm]{Definition}
 \theoremstyle{remark}
 \newtheorem{rem}[thm]{Remark}
 \newtheorem{ex}[thm]{Example}
 \theoremstyle{example}
 \numberwithin{equation}{section}
 \def\stackunder#1#2{\mathrel{\mathop{#2}\limits_{#1}}}
 \newcommand{\hooklongrightarrow}{\lhook\joinrel\longrightarrow}

 \title{DIRECTED WEAK FIBRATIONS }
 \author{I. POP}
 \date{}
 \maketitle
\begin{abstract}In a previous paper \cite{Pop} the author studied the homotopy lifting property in the category dTop of directed spaces in the sense of M. Grandis \cite{Grand1}, \cite{Grand2}, \cite{Grand3}. The present paper, which is a continuation of aforementioned article, introduces and studies the directed weak homotopy property (dWCHP) and directed weak fibrations, extending to the category dTop the well known Dold's (or weak) fibrations \cite{Dold1}. The dWCHP is characterized in several ways. Then the notion of a directed fibre homotopy equivalence (dFHE) between directed weak fibrations is studied. Some examples and counterexamples are given.


\end{abstract}

\section{Introduction}
Directed Algebraic Topology is a recent subject which arose from the study of some phenomena in the analysis of concurrency, traffic networks,space-time models, etc.(\cite{Faj1}-\cite{Gou1}). It was systematically developed by Marco Grandis (\cite{Grand1},\cite{Grand2},\cite{Grand3}). Directed spaces have privileged directions and directed paths therefore do not need to be reversible. M.Grandis introduced and studied 'non-reversible' homotopical tools corresponding to ordinary homotopies, fundamental group and fundamental $n$-groupoids: directed homotopies, fundamental monoids and fundamental $n$-categories. Also some directed homotopy constructions were considered: pushouts and pullbacks, mapping cones and homotopy fibres, suspensions and loops, cofibre and fibre spaces. As for directed fibrations, M.Grandis \cite{Grand2} refers to these (more precisely to the so-called lower and upper d-fibrations) only in relation with directed h-pullbacks. But in \cite{Pop},  the author of this paper defined and studied in detail Hurewicz directed fibrations (bilateral d-fibrations in the sense of definitions by Grandis). In classical algebraic topology, besides the homotopy covering property (CHP)\cite{Hur}, other properties of covering/lifting homotopy have also been studied and proved very interesting :weak CHP (WCHP)\cite{Dold1},\cite{Dieck},\cite{Kamp1}, \cite{James}; rather weak CHP (RWCHP)\cite{Brown1},\cite{Brown2},\cite{Pop1},\cite{Kie}; very weak CHP (VWCHP) \cite{Kie}; approximate fibrations (AHLP)\cite{CorDuv},\cite{CorDuv2},\cite{Pop2}. In this paper the author proposes to define and study a directed weak lifting homotopy property (dWHLP).

The basics of the Directed Algebraic Topology which we will use are taken from the 2003 paper of Grandis \cite{Grand1}.

A \textsl{directed space} or a \textsl{d-space}, is a topological space $X$ equipped with a set $dX$ of continuous maps $a:\mathbf{I}=[0,1]\rightarrow X$, called \textsl{directed paths}, or \textsl{d-paths}, satisfying the following three axioms:

(i)(constant paths) every constant map $\mathbf{I}\rightarrow X$ is a directed path,

(ii)(reparametrisation) the set $dX$ is closed under composition with (weakly)increasing maps $\mathbf{I}\rightarrow \mathbf{I}$ ,

(iii) (concatenation) the set $dX$ is closed under concatenation (the product of consecutive paths, which will be denoted by $\ast$).

We use the notations $\underline{X}$ or $\uparrow X$ if $X$ is the underlying topological space; if $\underline{X}$ (or $\uparrow X$) is given, then the set of directed paths is denoted by $d\underline{X}$ (resp.$d(\uparrow X)$)and the underlying space by $|\underline{X}|$ (resp.$|\uparrow X|$).

The standard $d$-interval with the directed paths given by increasing (weakly) maps $\mathbf{I}\rightarrow \mathbf{I}$ is denoted by $\uparrow \mathbf{I}=\uparrow [0,1]$.

A \textsl{directed map}, or $d-map$, $f:\underline{X}\rightarrow \underline{Y}$, is a continuous mapping between d-spaces which preserves the directed paths: if $a\in d\underline{X}$, then $f\circ a\in d\underline{Y}$. The category of directed spaces and directed maps is denoted by $d\mathbf{Top}$ (or $\uparrow \mathbf{Top}$). A directed path $a\in d\underline{X}$ defines a directed map $a:\uparrow \mathbf{I}\rightarrow \underline{X}$ which is also a \textsl{path} of $\underline{X}$.

For two points $x,x'\in \underline{X}$ we write $x\preceq x'$ if there exists a directed path from $x$ to $x'$. The equivalence relation $\simeq $ spanned by $\preceq $ yields the partition of a d-space $\underline{X}$ in its directed path components and a functor $\uparrow \Pi_0:d\mathbf{Top}\rightarrow \mathbf{Set}$, $\uparrow \Pi_0(\underline{X})=|\underline{X}|/\simeq $. A non-empty d-space $\underline{X}$ is \textsl{directed path connected} if $\uparrow \Pi_0(\underline{X})$ contains only one element.

The \textsl{directed cylinder} of a d-space $\underline{X}$ is the d-space $\uparrow (|\underline{X}|\times \mathbf{I})$ , denoted by $\underline{X}\times \uparrow \mathbf{I}$ or $\uparrow \mathbf{I}\underline{X}$, for which a path $\mathbf{I}\rightarrow |\underline{X}|\times \mathbf{I}$ is directed if and only if its components $\mathbf{I}\rightarrow \underline{X}$, $\mathbf{I}\rightarrow \mathbf{I}$ are directed. The directed maps $\partial^{\alpha}:\underline{X}\rightarrow  \underline{X}\times \uparrow \mathbf{I}$, $\alpha=0,1$, defined by $\partial^{\alpha}(x)=(x,\alpha)$, are called the \textsl{faces} of the cylinder.

If $f,g:\underline{X}\rightarrow \underline{Y}$ are directed maps, a \textsl{directed homotopy} $\varphi$  from $f$ to $g$, denoted by $\varphi:f\rightarrow g$, or $\varphi:f\preceq g$, is a d-map $\varphi:\underline{X}\times \uparrow \mathbf{I}\rightarrow \underline{Y}$ such that $\partial^0\circ \varphi=f$ and $\partial^1\circ \varphi=g$. The equivalence relation defined by the d-homotopy preorder $\preceq $ is denoted by $f\simeq_d g$ or simply $f\simeq g$ . This means that there exists a finite sequence $f\preceq f_1\succeq f_2\preceq f_3\succeq...g$.

Two d-spaces $\underline{X}$ and $\underline{Y}$ are \textsl{d-homotopy equivalent} if there exist d-maps $f:\underline{X}\rightarrow \underline{Y}$ and $g:\underline{Y}\rightarrow \underline{X}$ such that $g\circ f\simeq_d 1_{\underline{X}}$ and $f\circ g\simeq_d 1_{\underline{Y}}$.

\textsc{Standard models}. The spaces $\textbf{R}^n$, $\textbf{I}^n$, $\textbf{S}^n$  have their \textsl{natural} d-structure, admitting all (continuous) paths. $\textbf{I}$ is called the \textsl{natural} interval. The \textsl{directed real line}, or \textsl{d-line} $\uparrow \mathbf{R}$ is the Euclidean line with directed paths given by the increasing maps $\mathbf{I}\rightarrow \mathbf{R}$ (with respect to natural orders). Its cartesian power in d\textbf{Top}, the \textsl{n-dimensional real d-space }$\uparrow \mathbf{R}^n$ is similarly described (with respect to the product order, $x\leq x'$ iff $x_i\leq x'_i$ for all $i$). The \textsl{standard d-interval} $\uparrow \mathbf{I}=\uparrow [0,1]$ has the subspace structure of the d-line; the \textsl{standard d-cube } $\uparrow \mathbf{I}^n $ is its n-th power, and a subspace of $\uparrow \mathbf{R}^n$. The \textsl{standard directed circle} $\uparrow \textbf{S}^1$ is the standard circle with the \textsl{anticlockwise structure}, where the directed paths $a:\mathbf{I}\rightarrow \mathbf{S}^1$ move this way, in the plane: $a(t)=[1,\vartheta (t)]$, with an increasing function $\vartheta$ (in polar coordinates).

A \textsl{directed quotient} $\underline{X}/R$ has the quotient structure, formed of finite concatenations of projected d-paths; in particular, for a subset $A\subset |\underline{X}|$, by $\underline{X}/A$ is denoted the quotient of $\underline{X}$ which identifies all points of $A$. In particular, $\uparrow \mathbf{S}^n=(\uparrow \mathbf{I}^n/\partial \mathbf{I}^n), (n>0)$. The standard circle has another d-structure of interest, induced by $\mathbf{R}\times \uparrow \mathbf{R}$ and called the ordered circle $\uparrow \mathbf{O}^1\subset \mathbf{R}\times \uparrow \mathbf{R}$. It is the quotient of $\uparrow \mathbf{I}+\uparrow \mathbf{I}$ which identifies lower and upper endpoints, separately.

The forgetful functor $U:d\mathbf{Top}\rightarrow \mathbf{Top}$ has adjoint: at left  $c_0(X)$ with \textsl{d-discrete structure} of constant paths,  and right $C^0(X)$ with the \textsl{natural d-structure} of all paths.

Reversing d-paths, by involution $r:\mathbf{I}\rightarrow \mathbf{I}, r(t)=1-t$, gives the \textsl{reflected}, or opposite, d-space; this forms a (covariant) involutive endofunctor, called \textsl{reflection} $R:d\mathbf{Top}\rightarrow d\mathbf{Top}, R(\underline{X})=\underline{X}^{op},(a\in d(\underline{X }^{op})\Leftrightarrow a^{op}:=a\circ r\in d\underline{X}$). A d-space is \textsl{symmetric} if it is invariant under reflection. It is \textsl{reflexive}, or \textsl{self-dual}, if it is isomorphic to its reflection, which is more general. The d-spaces $\uparrow \mathbf{R}^n$,$\uparrow \mathbf{I}^n$, $\uparrow\mathbf{S}^n$ and $\uparrow \mathbf{O}^1$ are all reflexive.
\section{Directed fibrations}
In this section we resume sections 2 and 3 from \cite{Pop} including definitions and results that are necessary or interesting for this paper.
\begin{defn}\label{def.2.1}Let $p:\underline{E}\rightarrow \underline{B}$,$f:\underline{X}\rightarrow \underline{B}$ be directed maps. A d-map $f':\underline{X}\rightarrow \underline{E}$ is called a directed lift of $f$ with respect to $p$ if $p\circ f'=f$.
\end{defn}
\begin{defn}\label{def.2.2}A directed map $p:\underline{E}\rightarrow \underline{B}$ is said to have the directed homotopy lifting property with respect to a d-space $\underline{X}$ if, given d-maps $f':\underline{X}\rightarrow \underline{E}$ and $\varphi:\underline{X}\times \uparrow \mathbf{I}\rightarrow \underline{B}$, and $\alpha\in \{0,1\}$, such that $\varphi \circ \partial^{\alpha}=p\circ f'$, there is a directed lift of $\varphi$, $\varphi':\underline{X}\times \uparrow \mathbf{I}\rightarrow \underline{E}$, with respect to $p$, $p\circ \varphi'=\varphi$, such that $\varphi'\circ\partial^\alpha=f'$.
$$
\xymatrix{
\underline{E} \ar[rr]^p & & \underline{B}\\
\underline{X} \ar[u]^{f'}_{\ \  =} \ar[rr]_{\partial^{\alpha}} & &
\underline{X}\times \uparrow
\textbf{I}\ar[ull]^{\varphi'}\ar[u]_{\varphi}^{=\,\,\,\,\,\,\,} }
$$
A directed map $p:\underline{E}\rightarrow \underline{B}$ is called directed (Hurewicz) fibration if $p$ has the directed homotopy lifting property with respect to all directed spaces (dHLP). This property is also called the directed covering homotopy property (dCHP).
\end{defn}
\textsc{Properties of directed fibrations}.

P1. If $p:\underline{E}\rightarrow \underline{B}$ has the directed homotopy lifting property with respect to $\underline{X}$ and $f_0,f_1:\underline{X}\rightarrow \underline{B}$ are directed homotopic, $f_0\simeq_d f_1$, then $f_0$ has a directed lift with respect to $p$ iff $f_1$ has this property.

P2. Let $p:\underline{E}\rightarrow \underline{B}$ be a directed fibration and $a\in d\underline{B}$ with $a(\alpha)=p(e_\alpha),e_\alpha\in |\underline{E}|$, and $\alpha\in \{0,1\}$. Then there exists a directed path $a_\alpha\in d\underline{E}$ which is a lift of $a$, $p\circ a_\alpha=a$, with the $\alpha$-endpoint $e_\alpha$, $a_\alpha(\alpha)=e_\alpha$.

\textsc{Examples of directed fibrations}.

E1. Let $\underline{F}$ and $\underline{B}$  be arbitrary directed spaces and let $p:\underline{B}\times \underline{F}\rightarrow \underline{B}$ be the projection. Then $p$ is a directed fibration.

E2. Let $p:E\rightarrow |\underline{B}|$  be a Hurewicz fibration. For the space $E$ consider the maximal d-structure compatible with $d\underline{B}$ and $p$, i.e., $d(\uparrow E)=\{a\in E^I/p\circ a\in d\underline{B}\}$. Then $p:\uparrow E\rightarrow\underline{B}$ is a directed fibration.

E3. If $p:\underline{E}\rightarrow \underline{B}$ is a directed fibration, then the opposite map $p:\underline{E}^{op}\rightarrow \underline{B}^{op}$ is also a directed fibration.

For an intrinsic characterization of the dHLP we need some notations.

Given a d-map $p:\underline{E}\rightarrow \underline{B}$ and $\alpha\in \{0,1\}$, we consider the following d-subspace of the product (in d\textbf{Top}) $\underline{E}\times \underline{B}^{\uparrow \mathbf{I}}$
$$\underline{B}_\alpha=\{(e,a)\in \underline{E}\times \underline{B}^{\uparrow \mathbf{I}}|a(\alpha)=p(e)\} $$
(The d-structure of $\underline{B}^{\uparrow \mathbf{I}}$ is given by the exponential law, $d\mathbf{Top}(\uparrow \mathbf{I},d\mathbf{Top}(\uparrow \mathbf{I},\underline{B}))\approx d\mathbf{Top}(\uparrow \mathbf{I}\times \uparrow \mathbf{I},\underline{B})$,\cite{Grand1}).

\begin{defn}\label{def.2.3}
A directed lifting pair for a directed map $p:\underline{E}\rightarrow\underline{B}$ is a pair of d-maps

\begin{equation}
\lambda_\alpha:\underline{B}_\alpha\rightarrow \underline{E}^{\uparrow \mathbf{I}}, \alpha=0,1,
\end{equation}
satisfying the following conditions:

\begin{equation}
\lambda_\alpha(e,\omega)(\alpha)=e,
\end{equation}
\begin{equation}
p\circ \lambda_\alpha(e,\omega)=\omega,
\end{equation}
for each $(e,\omega)\in \underline{B}_\alpha$.
\end{defn}
\begin{thm}\label{thm.2.4}(i) A directed map $p:\underline{E}\rightarrow \underline{B}$ is a directed fibration if and only if there exists a directed lifting pair for $p$.

(ii) If $p:\underline{E}\rightarrow \underline{B}$ is a directed fibration, the d-spaces $\underline{B}_0$ and $\underline{B}_1$ are d-homotopiy equivalent.
\end{thm}

\section{Directed weak fibrations. Properties. Examples }
\begin{defn}\label{def.3.1}
Let $p:\underline{E}\rightarrow \underline{B}$ and $f,g:\underline{X} \rightarrow \underline{E} $ directed maps such that $p\circ f=p\circ g$. If we suppose that $f\simeq_d g$ by a sequence of directed homotopies  $f\preceq f_1\succeq f_2\preceq f_3\succeq...g$, $p\circ f_k=p\circ f=p\circ g$,  with the all homotopies $\varphi_1:f\preceq f_1,\varphi_2:f_2\preceq f_1...,\varphi_k:f_k\preceq f_{k\pm 1},....$ satisfying  the conditions $(p\circ\varphi_k)(x,t)=p(f(x))$ , $k=1,2,...$, $(\forall) x\in \underline{X},(\forall) t\in [0,1]$, then we denote this by $f\stackunder{(p)}{\simeq_d}g $ and say that $f$ is vertically directed homotopic to $g$.
\end{defn}
\begin{defn}\label{def.3.2}
A directed map $p:\underline{E}\rightarrow \underline{B}$ is said to have the directed weak covering homotopy property with respect to a d-space $\underline{X}$ if, given d-maps $ f':\underline{X}\rightarrow \underline{E}$ and $\varphi:\underline{X}\times \uparrow \mathbf{I}\rightarrow \underline{B}$, and $\alpha \in \{0,1\}$, such that $\varphi\circ \partial^\alpha=p\circ f'$, there is a directed lift of $\varphi$, $\varphi': \underline{X}\times \uparrow \mathbf{I}\rightarrow \underline{E}$, with respect to $p$, $p\circ \varphi'=\varphi$, such that $\varphi'\circ \partial^\alpha$ and $f'$ are vertically  directed homotopic $ \varphi'\circ \partial^\alpha \stackunder{(p)}{\simeq_d} f'$.
$$
\xymatrix{
\underline{E} \ar[rr]^p & & \underline{B}\\
\underline{X} \ar[u]^{f'}_{\ \  \stackunder{(p)}{\simeq_d}} \ar[rr]_{\partial^{\alpha}} & &
\underline{X}\times \uparrow
\textbf{I}\ar@{-->}[ull]^{\varphi'}\ar[u]_{\varphi}^{=\,\,\,\,\,\,\,} }
$$
\end{defn}
\begin{prop}\label{prop.3.3} If $p:\underline{E}\rightarrow \underline{B}$ has the directed weak covering homotopy property with respect to $\underline{X}$ and $f_0,f_1:\underline{X}\rightarrow \underline{B}$ are directed homotopic, $f_0\simeq_d f_1$, then $f_0$ has a direct lift with respect to $p$ if and only if $f_1$ has this property.
\end{prop}
\begin{proof}
Let $f'_0: \underline{X}\rightarrow \underline{E}$ be directed lifting of $f_0$, $p\circ f'_0=f_0$. If $\varphi:f_0\preceq f_1$ or $\varphi:f_1\preceq f_0$, let $\varphi\circ \partial^\alpha=f_0,\alpha\in \{0,1\}$. Then we have $\varphi\circ \partial^\alpha=p\circ f'_0$ and let $\varphi'$ a directed lift of $\varphi$, $p\circ \varphi'=\varphi$ and $ \varphi'\circ \partial^\alpha \stackunder{(p)}{\simeq_d} f_0'$. If we define $f'_1=\varphi'\circ \partial^{1-\alpha}$, then $p\circ f'_1=p\circ \varphi'\circ \partial^{1-\alpha}=\varphi\circ \partial^{1-\alpha}=f_1$. In the general case $f_0 \preceq g_1\succeq g_2\preceq g_3\succeq ...f_1$, we recurrently apply the consequences of the immediately homotopies.
\end{proof}
\begin{rem}\label{rem.3.4}
In the proof of Proposition \ref{prop.3.3} we used only a part from the definition of the directed weak covering homotopy property.
\end{rem}
\begin{defn}\label{def.3.5} A directed map $p:\underline{E}\rightarrow \underline{B}$ is called a directed weak fibration or a directed Dold fibration if $p$ has the directed weak covering homotopy property with respect to every directed space (dWCHP).
\end{defn}
\begin{cor}\label{cor.3.6}
Let $p:\underline{E} \rightarrow \underline{B}$ be a directed weak fibration and  $a\in d \underline{B}$ a directed path with $a(\alpha)=p(e_\alpha)$ for a point $e_\alpha \in |\underline{E}|$ and $\alpha\in \{0,1\}$. Then $a$ admits a directed lift $a'_\alpha\in d\underline{E}$, $p\circ a'_\alpha=a$, whose the de $\alpha-end$ point, $a'_\alpha(\alpha)$, is in same directed path component of $\underline{E}$ as $e_\alpha$.
\end{cor}
\begin{ex}\label{ex.3.7}Let $p:E\rightarrow B$ be a weak fibration and $\uparrow B$ a d-structure on the space $B$. For the space $E$ consider the maximal d-structure compatible with $p$, i.e., $a\in \uparrow E $ iff $p\circ a\in d(\uparrow B)$. Then the directed map $p:\uparrow E\rightarrow \uparrow B$ is a directed weak fibration. Indeed: if $f':\underline{X} \rightarrow\uparrow E$ and $\varphi:\underline{X}\times \uparrow \mathbf{I}\rightarrow \uparrow B$ are directed maps and $\alpha\in \{0,1\}$ such that $\varphi\circ \partial^\alpha =p\circ f'$, then obviously there exists continuous maps $\varphi':X\times I\rightarrow B$ and $\mathcal{K}:X\times I \rightarrow E$ satisfying $p\circ \varphi=p$, $\mathcal{K}(x,0)=f'(x),K(x,1)=\varphi'(x,\alpha)$ and $p\mathcal{K}(x,t)=pf'(x)$, $(\forall) x\in X, (\forall)t\in I$. Then it is enough to prove that the maps $\varphi'$ and $\mathcal{K}$ are directed. If $c\in d(\underline{X}\times \uparrow \mathbf{I})$, then $p\circ (\varphi'\circ c)=\varphi\circ c\in d(\uparrow B)$ such that $\varphi'\circ c\in d(\uparrow E)$, so that $\varphi'$ is directed. Then, if $c(x)=(c'(x),c''(x))$, we have $c'\in d(\underline{X})$, and $c''\in d(\uparrow \mathbf{I})$,  and we have $p\circ (\mathcal{K}\circ c)=p\circ (f'\circ c')$. And since $f'$ is directed, $f'\circ c'\in d(\uparrow E)$ and $p\circ (f'\circ c')\in d(\uparrow B)$ , hence $p\circ (K\circ c)\in d(\uparrow E)$ which implies $\mathcal{K}\circ c\in d(\uparrow B)$. Therefore we have $\mathcal{K}:f'\stackunder{(p)}{\preceq_d}\varphi'_\alpha$.
\end{ex}
\begin{ex}\label{ex.3.8}Consider in the directed space $\uparrow \mathbf{R}\times \mathbf{R}$ the subspaces $\underline{B}=\uparrow \mathbf{R}$ and $\underline{E}=\{(x,y)|x.y\geq 0\}$ , and the directed map $p:\underline{E}\rightarrow \underline{B}$ ,$p(x,y)=x$. This is not a directed fibration. Consider a pointed space $\{\ast\}$ and the directed maps $h':\{\ast\}\rightarrow \underline{E},h'(\ast)=(0,-1)$ and $H:{\ast}\times \uparrow I\rightarrow \underline{B}, H(\ast,t)=t$, for which $H(\ast,0)=ph'(\ast)=(0,0)$. For this pair there isn't a directed map $H':\{\ast\}\times \uparrow \mathbf{I}\rightarrow \underline{E}$ such that $H'(\ast,0)=(0,-1)$ and $p\circ H'=H$. If we suppose the contrary, let $H'(\ast,t)=(x(t),y(t))$, $x(t)y(t)\geq 0,x(0)=0,y(0)=-1$ and $x(t)=t,(\forall) t\in[0,1]$. So $y(t)\geq 0,(\forall)t\in I$, in contradiction with $y(0)=-1$. Therefore $p$ is not a directed fibration. But we can prove that it is a directed weak fibration. Suppose that for an arbitrary directed space $\underline{A}$ are given the directed maps $f':\underline{A}\rightarrow \underline{E}, \varphi:\underline{A}\times \uparrow \mathbf{I}\rightarrow \underline{B}$, and $\alpha\in \{0,1\}$, satisfying $p(f'(a))=\varphi(a,\alpha),(\forall)a\in |\underline{A}|$. If $f'(a)=(x(a),y(a)),x(a).y(a)\geq 0$, and $\varphi(a,\alpha)=x(a)$, we define $\varphi':A\times \uparrow \mathbf{I}$ by $\varphi'(a,t)=(\varphi(a,t),\varphi(a,t)$. This is a directed map and  $p\varphi'(a,t)=\varphi(a,t)$, i.e., $p\circ \varphi'=\varphi$. Then $\varphi'(a,\alpha)=(\varphi(a,\alpha),\varphi(a,\alpha))=(x(a),x(a))$. Now we define $\mathcal{K}:\underline{A}\times \uparrow \mathbf{I}\rightarrow \underline{E}$ by $\mathcal{K}(a,s)=(x(a),(1-s)y(a)+sx(a),s\in I,a\in |\underline{A}|$. This is correctly defined since $x(a).[(1-s)y(a)+sx(a)]\geq 0$. And it is a directed map since $f'$ is and therefore $x(.)$ is directed. For this map we have: $\mathcal{K}(a,0)=(x(a),y(a))=f'(a)$,$\mathcal{K(}a,1)=(x(a),x(a))=(\varphi'\circ\nabla^\alpha)(a)$ and  $p\mathcal{K}(a,s)=x(a)=p(f'(a)),(\forall)s\in I,a\in |\underline{A}|$. Therefore we have $\mathcal{K}:f'\stackunder{(p)}{\preceq_d}\varphi'\circ \partial^\alpha$.
\end{ex}
\begin{defn}\label{def.3.9}
Let $p:\underline{E}\rightarrow \underline{B}$, $p':\underline{E}'\rightarrow \underline{B}$ directed maps. We say that $p$ is directed dominated by $p'$ (or $p'$ dominates $p$) if there exist directed maps $f:\underline{E}\rightarrow \underline{E}',g:\underline{E}'\rightarrow \underline{E}$ over $\underline{B}$, $p'\circ f=p,p\circ g=p'$ such that $g\circ f\stackunder{(p)}{\simeq_d}id_{\underline{E}}$.
\end{defn}
\begin{thm}\label{thm.3.10}If $p:\underline{E}\rightarrow \underline{B}$ is directed dominated by $p':\underline{E}'\rightarrow \underline{B}$ and if $p'$ has the dWCHP with respect to $\underline{X}$, then $p$ has the same property. Consequently if $p'$ is a directed weak fibration then $p$ is also a directed weak fibration.
\end{thm}
\begin{proof}
Let some maps $f$ and $g$ as in Definition \ref{def.3.9}. Let $h':\underline{X}\rightarrow \underline{E},\varphi:\underline{X}\times \uparrow \mathbf{I}\rightarrow \underline{B}$, and $\alpha\in \{0,1\}$,  such that $\varphi\circ \partial^\alpha=p\circ h'$. We have the commutative diagram
$$
\xymatrix{
\underline{E}' \ar[rr]^{p'} & & \underline{B}'\\
\underline{X} \ar[u]^{f\circ h'} \ar[rr]_{\partial^{\alpha}} & &
\underline{X}\times \uparrow
\textbf{I}\ar[u]_{\varphi} }
$$

Then by hypothesis there is $\overline{\varphi'}:\underline{X}\times \uparrow \mathbf{I}\rightarrow \underline{E}'$, with $p'\circ \overline{\varphi'}=\varphi$ and $\overline{\varphi'}\circ \partial^\alpha \stackunder{(p')}{\simeq_d} f\circ h'$. Then we can consider the composition $\varphi'=g\circ \overline{\varphi'}$. For this we have $p\circ \varphi'=p\circ g\circ \overline{\varphi'}=\varphi$ and $\varphi'\circ \partial^\alpha=g\circ \overline{\varphi'}\circ \partial^\alpha \stackunder{(p)}{\simeq_d}(g\circ f)\circ h'\stackunder{(p)}{\simeq_d} h'$.
\end{proof}
\begin{cor}\label{cor.3.11}If $p:\underline{E}\rightarrow \underline{B}$ is directed dominated by a directed fibration $p':\underline{E}'\rightarrow \underline{B}$, then $p$ is a directed weak fibration.
\end{cor}
\begin{prop}\label{prop.3.12}Let $p:\underline{E}\rightarrow \underline{B}$ be a directed map with the dWCHP with respect to $\underline{X}$. Let $g':\underline{X}\rightarrow \underline{E}, \psi:\underline{X}\times (\uparrow \mathbf{I})^{op}\rightarrow \underline{B}$ directed maps, and $\alpha\in \{0,1\}$, such that $\psi\circ\partial^\alpha =p\circ g'$,where $\partial^\alpha$ denotes the faces of the inverse cylinder $\underline{X}\times (\uparrow \mathbf{I})^{op}$. Then there exists a directed map $\psi':\underline{X}\times (\uparrow \mathbf{I})^{op}\rightarrow \underline{E}$ satisfying the conditions $p\circ \psi'=\psi$ and $\psi'\circ \partial^\alpha \stackunder{(p)}{\simeq_d} g'$.
\end{prop}
\begin{proof}Define $\varphi:\underline{X}\times \uparrow \mathbf{I}\rightarrow \underline{B}$ by $\varphi(x,t)=\psi(x,1-t)$. This is a directed map. Indeed, if $c\in d(\underline{X}\times \uparrow \mathbf{I})$ we write $c(t)=(x(t),i(t))$, with $x(.)\in d\underline{X}$ and $i(.)\in d(\uparrow \mathbf{I})$ and define $c':I\rightarrow |\underline{X}|\times I$ by $c'(t)=(x(t),1-i(t))$. Since $i'(t)=1-i(t)$ defines a directed path in $(\uparrow \mathbf{I})^{op}$ , we have that $c'\in d(X\times (\uparrow \mathbf{I})^{op})$ and then $\psi\circ c'\in d(\underline{B})$. But $(\psi\circ c')(t)=\psi(x(t),1-i(t))=\psi(x(t),i(t))=(\varphi\circ c)(t)$. Thus $\varphi\circ c\in d(\underline{B})$ and so $\varphi $ is a directed map. Then by hypothesis there is a directed map $\varphi':\underline{X}\times \uparrow \mathbf{I}\rightarrow E$ satisfying $p\circ \varphi'=\varphi$ and $\varphi'\circ \partial^{1-\alpha}\stackunder{(p)}{\simeq_d} g'$. Now define $\psi':\underline{X}\times (\uparrow \mathbf{I})^{op}\rightarrow \underline{E}$ by $\psi'(x,t)=\varphi'(x,1-t)$. As above, this is a directed map. Moreover $p\psi'(x,t)=p\varphi'(x,1-t)=\varphi(x,1-t)=\psi(x,t)$. And $\psi'(x,\alpha)=\varphi'(x,1-\alpha)$, i.e., $\psi'\circ \partial^\alpha=\varphi'\circ \partial^{1-\alpha}$, thus we deduce $\psi'\circ \partial^\alpha\stackunder{(p)}{\simeq_d} g'$.
\end{proof}
\begin{cor}\label{cor.3.13}The reflection endofunctor $R:d\mathbf{Top}\rightarrow d\mathbf{Top}$ conserves the property of directed weak fibration.
\end{cor}
\begin{proof}
Let $p:\underline{E}\rightarrow \underline{B}$ be a directed weak fibration. Consider its image under the functor $R$, $p:\underline{E}^{op}\rightarrow \underline{B}^{op}$. For an arbitrary directed space $\underline{X}$, consider directed maps $f':\underline{X}\rightarrow \underline{E}^{op}$, $\varphi:\underline{X}\times \uparrow \mathbf{I}\rightarrow \underline{B}^{op}$ and $\alpha\in \{0,1\}$ such that $\varphi\circ \partial^\alpha=p\circ f'$. Consider the opposite of these maps $f':\underline{X}^{op}\rightarrow \underline{E}$ and $\varphi:(\underline{X}\times \uparrow \mathbf{I})^{op}\approx \underline{X}^{op}\times (\uparrow \mathbf{I})^{op}\rightarrow \underline{B}$ with $\varphi\circ \partial^\alpha=p\circ f'$ also. Then by Proposition \ref{prop.3.12} there is a commutative diagram
$$
\xymatrix{
\underline{E} \ar[rr]^p & & \underline{B}\\
\underline{X}^{op} \ar[u]^{f'}_{\ \  \stackunder{(p)}{\simeq_d}} \ar[rr]_{\partial^{\alpha}} & &
\underline{X}^{op}\times (\uparrow
\textbf{I})^{op}\ar[ull]^{\varphi'}\ar[u]_{\varphi}^{=\,\,\,\,\,\,\,} }
$$
and the functor $R$ induces the diagram
$$
\xymatrix{
\underline{E}^{op} \ar[rr]^p & & \underline{B}^{op}\\
\underline{X} \ar[u]^{f'}_{\ \  \stackunder{(p)}{\simeq_d}} \ar[rr]_{\partial^{\alpha}} & &
\underline{X}\times \uparrow
\textbf{I}\ar[ull]^{\varphi'}\ar[u]_{\varphi}^{=\,\,\,\,\,\,\,} }
$$
since the reflection functor translates a relation $f\preceq g$ into $g\preceq f$, that is the homotopies are conserved.
\end{proof}
\begin{prop}\label{prop.3.14}Let $p:\underline{E}\rightarrow \underline{B}$ be a directed weak fibration and $f:\underline{B}'\rightarrow \underline{B}$ a directed map. Denote the pullback $\underline{E}\prod_{\underline{B}}\underline{B}'$ by $\underline{E}_f$, i.e., $\underline{E}_f=\{(e,b')\in \underline{E}\times \underline{B}'|p(e)=f(b')\}$ with the d-structure as subspace of the product $\underline{E}\times \underline{B}'$. Then the projection $p_f:\underline{E}_f\rightarrow \underline{B}'$ is a directed weak fibration.
\end{prop}
\begin{proof}Let be given the commutative diagram
$$
\xymatrix{
\underline{E}_f \ar[rr]^{p_f} & & \underline{B}'\\
\underline{X} \ar[u]^{h'} \ar[rr]_{\partial^{\alpha}} & &
\underline{X}\times \uparrow
\textbf{I}\ar[u]_{\varphi} }
$$
such that $h'(x)=(e(x),b'(x))$, $p(e(x))=f(b'(x))$, with $e(.):\underline{X}\rightarrow \underline{E}$ and $b'(.):\underline{X}\rightarrow \underline{B}'$ directed maps, we have $\varphi(x,\alpha)=b'(x)$ ,$(\forall) x\in \underline{X}$. Now we can consider the following commutative diagram
$$
\xymatrix{
\underline{E} \ar[rr]^p & & \underline{B}\\
\underline{X} \ar[u]^{e(.)} \ar[rr]_{\partial^{\alpha}} & &
\underline{X}\times \uparrow
\textbf{I}\ar[u]_{f\circ \varphi} }
$$
because $p(e(x))=f(b'(x))=f(\varphi(x,\alpha))=((f\circ \varphi)\circ \partial^\alpha)(x)$.

Now, by hypothesis there is a directed map $\overline{\varphi}':\underline{X}\times \uparrow \mathbf{I}\rightarrow \underline{E}$ such that $p\circ \overline{\varphi}'=f\circ \varphi$ and $\overline{\varphi}'\circ \partial^\alpha \stackunder{(p)}{\simeq_d} e(.)$.
$$
\xymatrix{
\underline{E} \ar[rr]^p & & \underline{B}\\
\underline{X} \ar[u]^{e(.)}_{\ \  \stackunder{(p)}{\simeq_d}} \ar[rr]_{\partial^{\alpha}} & &
\underline{X}\times \uparrow
\textbf{I}\ar[ull]^{\overline{\varphi}'}\ar[u]_{f\circ\varphi}^{=\,\,\,\,\,\,\,} }
$$
From the relation $p(\overline{\varphi}'(x,t))=f(\varphi(x,t))$ we see that for all $(x,t)\in \underline{X}\times \uparrow \mathbf{I}$,  $(\varphi'(x,t),\varphi(x,t))\in \underline{E}_f$, such that we can define $\varphi':\underline{X}\times \uparrow \mathbf{I}\rightarrow \underline{E}_f$ by $\varphi'(x,t)=\overline{\varphi}'(x,t),\varphi(x,t))$. This is a directed map and $(p_f\circ \varphi')(x,t)=\varphi(x,t)$, i.e., $p_f\circ \varphi'=\varphi$. Moreover, $(\varphi'\circ \partial^\alpha)(x)=((\overline{\varphi}'\circ \partial^\alpha)(x),(\varphi\circ \partial^\alpha)(x))=((\overline{\varphi}'\circ \partial^\alpha)(x),(p_f\circ h')(x))=((\overline{\varphi}'\circ \partial^\alpha)(x),b'(x))$, i.e., $\varphi'\circ \partial^\alpha=(\overline{\varphi}'\circ \partial^\alpha,b'(.))\stackunder{(p)}{\simeq_d}(e(.),b'(.))=h'$,

$$
\xymatrix{
\underline{E}_f \ar[rr]^{p_f} & & \underline{B}'\\
\underline{X} \ar[u]^{h'}_{\ \  \stackunder{(p)}{\simeq_d}} \ar[rr]_{\partial^{\alpha}} & &
\underline{X}\times \uparrow
\textbf{I}\ar[ull]^{\varphi'}\ar[u]_{\varphi}^{=\,\,\,\,\,\,\,} }
$$
\end{proof}
\section{A further characterization}
\begin{defn}\label{def.4.1}A directed homotopy $\varphi:\underline{X}\times \uparrow \mathbf{I}\rightarrow \underline{Y}$ is called semistationary if either  $\varphi(x,t)=\varphi(x,\frac{1}{2}),(\forall)x\in X,t\in [0,\frac{1}{2}]$ or $\varphi(x,t)=\varphi(x,\frac{1}{2}),(\forall) x\in X,t\in [\frac{1}{2},1]$. In the first case we say that $\varphi$ is lower semistationary and in the second case we say that $\varphi$ is upper semistationary.
 \end{defn}
\begin{thm}\label{thm.4.2}
A directed map $p:\underline{E}\rightarrow \underline{B}$ has the directed weak covering homotopy property (dWCHP) with respect to a directed space $\underline{X}$ if and only if $p$ has the directed covering homotopy property(dCHP) with respect to $\underline{X}$ for all semistationary directed homotopies.
\end{thm}
\begin{proof}$\Leftarrow$  Suppose that $p$ has the dCHP with respect to $\underline{X}$ for all semistationary directed homotopies and let be given a commutative diagram
$$
\xymatrix{
\underline{E} \ar[rr]^p & & \underline{B}\\
\underline{X} \ar[u]^{f'} \ar[rr]_{\partial^{\alpha}} & &
\underline{X}\times \uparrow
\textbf{I}\ar[u]_{ \varphi} }
$$

a)Suppose $\alpha=0$. Then we define the following lower semistationary directed homotopy: $\overline{\varphi}_{-}:\underline{X}\times \uparrow \mathbf{I} \rightarrow \underline{B}$,
\[\overline{\varphi}_{-}(x,t)=\left\{ \begin{array}{ll}
\varphi(x,0), & \mbox{if $0\leq t\leq \frac{1}{2} $},\\
\varphi(x,2t-1), & \mbox{if $\frac{1}{2}\leq t\leq 1 $}.\end{array} \right.
\]
This map is well defined and directed because the map $\theta:\uparrow \mathbf{I}\rightarrow \uparrow \mathbf{I}$ defined by $\theta(t)=0,t\leq \frac{1}{2}$ and $\theta(t)=2t-1,\frac{1}{2}\leq t\leq 1$ is obviously directed. Then $\overline{\varphi}_{-}$ is lower semistationary, $\overline{\varphi}_{-}(x,t)=pf'(x),(\forall) x\in \underline{X}, t\in [0,\frac{1}{2}]$. By hypothesis there is a homotopy $\overline{\varphi}'_{-}:\underline{X}\times \uparrow \mathbf{I}\rightarrow \underline{E}$ with $\overline{\varphi}'_{-}\circ \partial^0=f'$ and $p\circ \overline{\varphi}'_{-}=\overline{\varphi}_{-}$, such that

\[p\overline{\varphi}'_{-}(x,t)=\overline{\varphi}_{-}(x,t)=\left\{ \begin{array}{ll}
\varphi(x,0)=pf'(x), & \mbox{if $0\leq t\leq \frac{1}{2} $},\\
\varphi(x,2t-1), & \mbox{if $\frac{1}{2}\leq t\leq 1 $}.\end{array} \right.
\]
Define $\varphi':\underline{X}\times \uparrow \mathbf{I}\rightarrow \underline{E}$ by $\varphi'(x,t)=\overline{\varphi}'_{-}(x,\frac{1+t}{2})$ which is a directed map and satisfies $p\varphi'(x,t)=p\overline{\varphi}'_{-}(x,\frac{1+t}{2})=\varphi(x,t)$ and $\varphi'(x,0)=\overline{\varphi}'_{-}(x,\frac{1}{2})$ and $f'(x)=\overline{\varphi}'_{-}(x,0)$.
Now we define $\hbar:\underline{X}\times \uparrow \mathbf{I}\rightarrow \underline{E}$ , by $\hbar(x,t)=\overline{\varphi}'_{-}(x,\frac{t}{2})$. For this directed homotopy we have $\hbar(x,0)=\overline{\varphi}'_{-}(x,0)=f'(x)$, $\hbar(x,1)=\overline{\varphi}'_{-}(x,\frac{1}{2})=\varphi'_0(x)$ and $p\hbar(x,t)=p\overline{\varphi}'_{-}(x,\frac{t}{2})=\overline{\varphi}_{-}(x,\frac{t}{2})=\varphi(x,0)=pf'(x)=p\varphi'_0$. Therefore $\hbar:f'\stackunder{(p)}{\preceq_d}\varphi'_0$.

b) Suppose $\alpha=1$. Then we define the following upper semistationary directed homotopy $\overline{\varphi}_{+}:\underline{X}\times \uparrow \mathbf{I} \rightarrow \underline{B}$,
\[\overline{\varphi}_{+}(x,t)=\left\{ \begin{array}{ll}
\varphi(x,2t), & \mbox{if $0\leq t\leq \frac{1}{2} $},\\
\varphi(x,1), & \mbox{if $\frac{1}{2}\leq t\leq 1 $}.\end{array} \right.
\]
For this we have $\overline{\varphi}_{+}(x,t)=\varphi(x,1)=pf'(x)$, $(\forall) x\in \underline{X}, t\in [\frac{1}{2},1]$. By hypothesis there is a homotopy $\overline{\varphi}'_{+}:\underline{X}\times \uparrow \mathbf{I}\rightarrow \underline{E}$ with $\overline{\varphi}'_{+}\circ \partial^1=f'$ and $p\circ \overline{\varphi}'_{+}=\overline{\varphi}_{+}$, such that

\[p\overline{\varphi}'_{+}(x,t)=\overline{\varphi}_{+}(x,t)=\left\{ \begin{array}{ll}
\varphi(x,2t), & \mbox{if $0\leq t\leq \frac{1}{2} $},\\
\varphi(x,1)=pf'(x), & \mbox{if $\frac{1}{2}\leq t\leq 1 $}.\end{array} \right.
\]
Define $\varphi':\underline{X}\times \uparrow \mathbf{I}\rightarrow \underline{E}$ by $\varphi'(x,t)=\overline{\varphi}'_{+}(x,\frac{t}{2})$ which is a directed map and satisfies $p\varphi'(x,t)=p\overline{\varphi}'_{+}(x,\frac{t}{2})=\varphi(x,t)$ and $\varphi'(x,1)=\overline{\varphi}'_{+}(x,\frac{1}{2})$ and $f'(x)=\overline{\varphi}'_{+}(x,1)$.

Now we define $\overline{\hbar}:\underline{X}\times \uparrow \mathbf{I}\rightarrow \underline{E}$ , by $\overline{\hbar}(x,t)=\overline{\varphi}'_{+}(x,\frac{1+t}{2})$. For this directed homotopy we have $\overline{\hbar}(x,0)=\overline{\varphi}'_{+}(x,\frac{1}{2})=(\varphi'\circ \partial^1)(x)$, $\overline{\hbar}(x,1)=\overline{\varphi}'_{+}(x,1)=f'(x)$, and $p\overline{\hbar}(x,t)=p\overline{\varphi}'_{+}(x,\frac{1+t}{2})=\overline{\varphi}_{+}(x,\frac{1+t}{2})=\varphi(x,1)=pf'(x)=p\varphi'_1$. Therefore $\overline{\hbar}:\varphi'\circ \partial^1\stackunder{(p)}{\preceq_d}f'$. This finishes the proof of the implication $\Leftarrow$.

The proof of the reverse implication requires the following two lemmas whose proof is immediate.
\begin{lem}\label{lem.4.3}Let $p:\underline{E}\rightarrow \underline{B}$ be a directed map satisfying the directed weak homotopy lifting property with respect to $\underline{X}$. Consider a directed interval $\uparrow [a,b]\subset \uparrow \mathbf{R},a<b $, and suppose be given some directed maps $\overline{f}':\underline{X}\rightarrow \underline{E},\overline{\varphi}:\underline{X}\times \uparrow \mathbf{I}\rightarrow \underline{B}$ such that $\overline{\varphi}_{\beta}=p\circ f'$, with $\overline{\varphi}_{\beta}(x)=\overline{\varphi}(x,\beta),\beta\in \{a,b\}$ . Then there exists $\overline{\varphi}':\underline{X}\times \uparrow [a,b]\rightarrow \underline{E}$ satisfying $p\circ \overline{\varphi}'=\overline{\varphi}$ and $\overline{\varphi}'_{\beta}\stackunder{(p)}{\simeq_d}\overline{f}'$.
\end{lem}
\begin{lem}\label{lem.4.4}Let be given directed maps $p:\underline{E}\rightarrow \underline{B}$, $f,g:\underline{X}\rightarrow \underline{E}$,with $p\circ f=p\circ g$ and $f\stackunder{(p)}{\simeq_d}g$. Then for every directed interval $\uparrow [a,b]\subset \uparrow \mathbf{R}, a<b$ there exists a directed map $G:\underline{X}\times \uparrow [a,b]\rightarrow \underline{E}$ satisfying the conditions:$G(x,a)=f(x),G(x,b)=g(x),pG(x,s)=pf(x)=pg(x),(\forall x\in \underline{X},s\in [a,b])$.
\end{lem}
$\Rightarrow $ Suppose that $p:\underline{E}\rightarrow \underline{B}$ satisfies the directed weak homotopy lifting property with respect to $X$ and let be given a commutative diagram
$$
\xymatrix{
\underline{E} \ar[rr]^p & & \underline{B}\\
\underline{X} \ar[u]^{f'} \ar[rr]_{\partial^{0}} & &
\underline{X}\times \uparrow
\textbf{I}\ar[u]_{\varphi} }
$$
with $\varphi$ a lower semistationary directed homotopy, $\varphi(x,t)=pf'(x),(\forall) x\in \underline{X},t\in [0,\frac{1}{2}]$.
Consider the commutative diagram
$$
\xymatrix{
\underline{E} \ar[rr]^p & & \underline{B}\\
\underline{X} \ar[u]^{f'} \ar[rr]_{j_{1/2}} & &
\underline{X}\times \uparrow
[1/2,1]\ar[u]_{\widetilde{\varphi}=\varphi|\underline{X}\times  \uparrow[1/2,1]} }
$$
with $j_{1/2}(x)=(x,1/2)$. Applying Lemma \ref{lem.4.3}, there exists a directed map $\widetilde{\varphi}':\underline{X}\times \uparrow [\frac{1}{2},1]\rightarrow \underline{E}$, with $p\circ \widetilde{\varphi}'=\widetilde{\varphi}=\varphi|X\times \uparrow [\frac{1}{2},1]$ and $\widetilde{\varphi}'_{1/2}\stackunder{(p)}{\simeq_d}f'$.  Now apply Lemma \ref{lem.4.4} for the maps $f',\widetilde{\varphi}'_{1/2}\underline{X}\rightarrow \underline{E}$, with $\uparrow [a,b]= \uparrow[0,\frac{1}{2}]$. Then there is a directed map $\kappa:\underline{X}\times \uparrow [0,\frac{1}{2}]\rightarrow \underline{E}$, with $\kappa_0=f',\kappa_{1/2}=\widetilde{\varphi}'_{1/2}$ and $p\kappa(x,t)=pf'(x),(\forall) x\in \underline{X},t\in [0,\frac{1}{2}]$. Now we can define $\varphi'_{-}:\underline{X}\times \uparrow \mathbf{I}\rightarrow \underline{E}$ by

\[\varphi'_{-}(x,t)=\left\{ \begin{array}{ll}
\kappa(x,t), & \mbox{if $0\leq t\leq 1/2 $},\\
\widetilde{\varphi}'(x,t), & \mbox{if $1/2\leq t\leq 1 $}.\end{array} \right.
\]
This satisfies $\varphi'_{-}(x,0)=\kappa(x,0)=f'(x)$ and
\[p\varphi'_{-}(x,t)=\left\{ \begin{array}{ll}
p\kappa(x,t)=ph'(x)=\varphi(x,t), & \mbox{if $0\leq t\leq \frac{1}{2} $},\\
p\widetilde{\varphi}'(x,t)=\widetilde{\varphi}(x,t)=\varphi(x,t), & \mbox{if $\frac{1}{2}\leq t\leq 1 $}.\end{array} \right.
\]
If $\varphi$ is an upper semistationary directed homotopy, $\varphi(x,t)=pf'(x),(\forall)x\in \underline{X},t\in [\frac{1}{2},]$, then consider the commutative diagram
$$
\xymatrix{
\underline{E} \ar[rr]^p & & \underline{B}\\
\underline{X} \ar[u]^{f'} \ar[rr]_{j_{1/2}} & &
\underline{X}\times \uparrow
[0,1/2]\ar[u]_{\overline{\varphi}=\varphi|\underline{X}\times  \uparrow[0,1/2]} }
$$

Applying Lemma \ref{lem.4.3} there exists $\overline{\varphi}':\underline{X}\times \uparrow [0,\frac{1}{2}]\rightarrow \underline{E}$, with $p\circ \overline{\varphi}'=\overline{\varphi}=\varphi|\underline{X}\times \uparrow [0,\frac{1}{2}]$ and $\overline{\varphi}'_{1/2}\stackunder{(p)}{\simeq_d}f'$. Now applying Lemma \ref{lem.4.4} on the directed interval $\uparrow [a,b]=\uparrow [\frac{1}{2},1]$, for the maps $f'$ and $\overline{\varphi}'_{1/2}$, there exists $\overline{\kappa}:\underline{X}\times \uparrow [0,\frac{1}{2}]\rightarrow \underline{E}$, with $\overline{\kappa}_1=f',\overline{\kappa}_{1/2}=\overline{\varphi}'_{1/2}$ and $p\overline{\kappa}(x,t)=pf'(x),(\forall) x\in \underline{X},t\in [\frac{1}{2},1]$. Finally, we can define $\varphi'_{+}:\underline{X}\times \uparrow \mathbf{I}\leftarrow \underline{E}$, by

\[\varphi'_{+}(x,t)=\left\{ \begin{array}{ll}
\overline{\varphi}'(x,t), & \mbox{if $0\leq t\leq \frac{1}{2} $},\\
\overline{\kappa}(x,t), & \mbox{if $\frac{1}{2}\leq t\leq 1 $}.\end{array} \right.
\]
For this we have $\varphi'_{+}(x,1)=\overline{\kappa}(x,1)=f'(x)$ and
\[p\varphi'_{+}(x,t)=\left\{ \begin{array}{ll}
p\overline{\varphi}'(x,t)=\varphi(x,t), & \mbox{if $0\leq t\leq \frac{1}{2} $},\\
p\overline{\kappa}(x,t)=pf'(x)=\varphi(x,t), & \mbox{if $\frac{1}{2}\leq t\leq 1 $}.\end{array} \right.
\]
\end{proof}
\begin{cor}\label{cor.4.5}Let $p:E\rightarrow B$ be a directed weak fibration. Let $f':\underline{X}\rightarrow \underline{E}$ a d-map.
Suppose that for $\alpha\in \{0,1\}$ there is a map $\varphi:\underline{X}\times \uparrow \mathbf{I}\rightarrow \underline{B}$ such that
$\varphi\circ \partial^\alpha =p\circ f'$. Also suppose $\varepsilon \in (0,1)$  be chosen.

(i) If $\alpha=0$ and $\varphi $ is stationary on the interval $[0,\varepsilon]$, then there exists $\varphi':\underline{X}\times \uparrow \mathbf{I}
\rightarrow \underline{E}$ satisfying the relations $\varphi'\circ \partial^0=f'$ and $p\circ \varphi'=\varphi$.

(ii) If $\alpha=1$ and $\varphi$ is stationary on the interval $[\varepsilon,1]$, then there exists $\varphi':\underline{X}\times \uparrow \mathbf{I}
\rightarrow \underline{E}$ satisfying the relations $\varphi'\circ \partial^1=f'$ and $p\circ \varphi'=\varphi$.
\end{cor}
\begin{proof}(i) Consider the directed isomorphism $\theta:\uparrow \mathbf{I}\rightarrow \uparrow \mathbf{I}$ defined by

 \[\theta(t)=\left\{ \begin{array}{ll}
 2\varepsilon t, & \mbox{if $0\leq t\leq \frac{1}{2} $},\\
2(1-\varepsilon)t+2\varepsilon-1, & \mbox{if $\frac{1}{2}\leq t\leq 1 $}.\end{array} \right.
 \]
 Using this isomorphism, transfer the homotopy $\varphi$ into a lower semistationary homotopy $\widetilde{\varphi}:\underline{X}\times \uparrow
 \mathbf{I}\rightarrow \underline{B}$ by defining $\widetilde{\varphi}=\varphi\circ \theta$. Indeed, if $t\in [0,\frac{1}{2}],\widetilde{\varphi}(x,t)=
 \varphi(x,2\varepsilon t)$, with $2\varepsilon t\in [0,\varepsilon]$, such that $\widetilde{\varphi(}x,t)=\widetilde{\varphi}(x,0)=\varphi(x,0)=pf'(x)$. For this, by Theorem \ref{thm.4.2}, there exists $\widetilde{\varphi}':\underline{X}\times \uparrow \mathbf{I}\rightarrow \underline{E}$ such that $\widetilde{\varphi}'\circ \partial^0=f'$ and
 $p\circ\widetilde{\varphi}'=\widetilde{\varphi}$ . Now define $\varphi':\underline{X}\times \uparrow \mathbf{I}\rightarrow \underline{E}$ by
 $\varphi'=\widetilde{\varphi}'\circ \theta^{-1}$, i.e.,

 \[\varphi'(x,t)=\left\{ \begin{array}{ll}
 \widetilde{\varphi}'(x,\frac{t}{2\varepsilon}), & \mbox{if $0\leq t\leq \varepsilon $},\\
 \widetilde{\varphi}'(x,\frac{t+1-2\varepsilon}{2(1-\varepsilon)}, & \mbox{if $\varepsilon\leq t\leq 1 $}.\end{array} \right.
 \]
 For this we have $\varphi'(x,0)=\widetilde{\varphi}'(x,0)=f'(x)$ and $p\circ \varphi'=p\circ \widetilde{\varphi}'\circ \theta^{-1}=
 \widetilde{\varphi}\circ \theta^{-1}=\varphi\circ \theta\circ\theta^{-1}=\varphi$.(Attention, $\theta^{-1}$ is a d-map! This is not the inverse path of $\theta$ ).

 (ii) We proceed similarly to (i).If $\varphi$ is stationary on $[\varepsilon,1]$, then $\widetilde{\varphi}=\varphi\circ \theta$ is upper semistationary,because for $t\in [\frac{1}{2},1]$ $\widetilde{\varphi}(x,t)=\varphi(x,2(1-\varepsilon)t+2\varepsilon -1)$, with
 $2(1-\varepsilon)t+2\varepsilon -1\in [\varepsilon,1]$, such that $\widetilde{\varphi}(x,t)=\widetilde{\varphi}(x,1)=\varphi(x,1)=pf'(x)$.
 Then continue as in (i).
 \end{proof}

\section{An intrinsic characterization: directed semistationary lifting pair}
Given a d-map $p:\underline{E}\rightarrow \underline{B}$ and $\alpha\in\{0,1\}$, we consider the following d-subspaces of the directed product $\underline{E}\times \underline{B}^{\uparrow \mathbf{I}}$:

$$\underline{B}^s_0=\{(e,\omega)\in \underline{E}\times \underline{B}^{\uparrow \mathbf{I}}|\omega(t)=p(e),(\forall)t\in [0,\frac{1}{2}]\},$$
and
$$ \underline{B}^s_1=\{(e,\omega)\in \underline{E}\times \underline{B}^{\uparrow \mathbf{I}}|\omega(t)=p(e),(\forall)t\in [\frac{1}{2},1]\}.$$
(The d-structure of $\underline{B}^{\uparrow \mathbf{I}}$ is given by the exponential law, $d\textbf{Top}(\uparrow \mathbf{I},d\mathbf{Top}(\uparrow \mathbf{I},\underline{B}))\approx d\mathbf{Top}(\uparrow \mathbf{I}\times \uparrow \mathbf{I},\underline{B})$,\cite{Grand1}).
\begin{defn}\label{def.5.1}A directed semistationary lifting pair for a directed map $p:\underline{E} \rightarrow \underline{B}$ consists of a pair of d-maps
\begin{equation}
\lambda^s_\alpha:\underline{B}^s_\alpha\rightarrow \underline{E}^{\uparrow \mathbf{I}},\alpha=0,1,
\end{equation}
satisfying the following conditions:
\begin{equation}
\lambda^s_\alpha(e,\omega)(\alpha)=e,
\end{equation}
\begin{equation}
p\circ \lambda^s_\alpha(e,\omega)=\omega,
\end{equation}
for each $(e,\omega)\in \underline{B}^s_\alpha.$
\end{defn}
\begin{thm}\label{thm.5.2}A directed map $p:\underline{E}\rightarrow \underline{B}$ is a directed weak fibration if and only if there exists a directed semistationary lifting pair for $p$.
\end{thm}
\begin{proof}Suppose that $p$ is a directed weak fibration. Then by Theorem \ref{thm.4.2},  $p$ has the dHLP with respect to all the directed spaces for directed semistationary homotopies.

For $\alpha\in\{0,1\}$, define the maps $f'_\alpha:\underline{B}^s_\alpha\rightarrow \underline{E}$ and $\varphi_\alpha:\underline{B}^s_\alpha\times \uparrow \mathbf{I}\rightarrow \underline{B}$, by $f'_\alpha(e,\omega)=e$ and $\varphi((e,\omega),t)=\omega(t)$. These are continuous and directed maps by the exponential laws. Moreover, since for every pair $(e,\omega)\in \underline{B}^s_\alpha$, $\omega$ is a semistationary path, we have that $\varphi_\alpha$ is a semistationary homotopy. If $\alpha=0$ then for all $t\in [0,\frac{1}{2}],\varphi_0((e,\omega),t)=\omega(t)=p(e)=(p\circ f'_0)(e,\omega)$ and if  $\alpha=1$ the for all $t\in [\frac{1}{2},1], \varphi_1(e,\omega)(t)=\omega(t)=p(e)=(p\circ f'_1)(e,\omega)$. Then applying the dHLP for the commutative diagram

$$
\xymatrix{
\underline{E} \ar[rr]^p & & \underline{B}\\
\underline{B}^s_\alpha \ar[u]^{f'_\alpha} \ar[rr]_{\partial^{\alpha}} & &
\underline{B}^s_\alpha\times \uparrow
\textbf{I}\ar[u]_{\varphi_\alpha} }
$$
there exists $\varphi'_\alpha:\underline{B}^s_\alpha\times \uparrow \mathbf{I}\rightarrow \underline{E}$, with $p\circ \varphi'_\alpha=\varphi_\alpha$  and $\varphi'_\alpha\circ \partial^\alpha=f'_\alpha$. Then we can define $\lambda^s_\alpha:\underline{B}^s_\alpha\rightarrow \underline{E}^{\uparrow \mathbf{I}}$, by $\lambda^s_\alpha(e,\omega)(t)=\varphi'((e,\omega),t)$. For this we have $(p\circ\lambda^s_\alpha(e,\omega))(t)=p(\varphi'((e,\omega),t))=\varphi_\alpha((e,\omega),t)=\omega(t)$ and $\lambda^s_\alpha(e,\omega)(\alpha)=\varphi'_\alpha((e,\omega),\alpha)=(\varphi'_\alpha\circ \partial^\alpha)(e,\omega)=f'_\alpha(e,\omega)=e$. Therefore $(\lambda^s_\alpha)_{\alpha=0,1}$ is a directed semistationary lifting pair for $p$.

Conversely, if $(\lambda^s_\alpha)_{\alpha=0,1}$ is a directed semistationary lifting pair for $p$, and for $\alpha\in\{0,1\}$ fixed, are given the d-maps $f':\underline{X} \rightarrow \underline{E}$ and $\varphi:\underline{X}\times \uparrow \mathbf{I}\rightarrow \underline{B}$, with $\varphi$ semistationary and $\varphi\circ \partial^\alpha=p\circ f'$, then we consider the directed map $g:\underline{X}\rightarrow \underline{B}^{\uparrow \mathbf{I}}$ defined by $g(x)(t)=\varphi(x,t)$. Since $g(x)$ is a semistationary directed path, with $g(x)(\alpha)=p(f'(x))$, we can define $\varphi':\underline{X}\times  \uparrow \mathbf{I}\rightarrow \underline{E}$, by $\varphi'(x,t)=\lambda^s_\alpha(f'(x),g(x))(t)$. For this we have $p\varphi'(x,t)=p\lambda^s_\alpha(f'(x),g(x))(t)=g(x)(t)=\varphi(x,t)$, and $\varphi'(x,\alpha)=\lambda^s_\alpha(f'(x),g(x))(\alpha)=g(x)(\alpha)=\varphi(x,\alpha)=f'(x)$. And by this we finished the proof.
\end{proof}
\begin{cor}\label{cor.5.3}Let $p:\underline{E}\rightarrow \underline{B}$ be a directed weak fibration . For  $\varepsilon\in (0,1)$ consider the
following directed subspaces of $\underline{E}\times \underline{B}^{\uparrow \mathbf{I}}$:

$\underline{B}_{\varepsilon}=\{(e,\omega)\in \underline{E}\times \underline{B}^{\uparrow \mathbf{I}} |\ \omega(t)=p(e),(\forall)t\in [0,\varepsilon]\}$ and

$\underline{B}^{\varepsilon}=\{(e,\omega)\in \underline{E}\times \underline{B}^{\uparrow \mathbf{I}} |\ \omega(t)=p(e),(\forall)t\in [\varepsilon,1]\}$.

Then there exists a pair of directed maps $\lambda_{\varepsilon}:\underline{B}_{\varepsilon}\rightarrow \underline{E}^{\uparrow \mathbf{I}}$ and
$\lambda^{\varepsilon}:\underline{B}^{\varepsilon}\rightarrow \underline{E}^{\uparrow \mathbf{I}}$, satisfying $\lambda_\varepsilon(e,\omega)(0)=e$,
$p\circ \lambda_\varepsilon(e,\omega)=\omega$ and, respectively $\lambda^{\varepsilon}(e',\omega')(1)=e',p\circ \lambda^{\varepsilon}(e',\omega')=\omega'$.
\end{cor}
\begin{proof}
For a path $\omega\in d\underline{B}$, consider $\widetilde{\omega}=\omega\circ \theta$, i.e.,
\[\widetilde{\omega}(t)=\left\{ \begin{array}{ll}
\omega(2\varepsilon t), & \mbox{if $0\leq t\leq \frac{1}{2} $},\\
\omega(2(1-\varepsilon)t+2\varepsilon-1), & \mbox{if $\frac{1}{2}\leq t\leq 1 $}.\end{array} \right.
\]
If $(e,\omega)\in \underline{B}_{\varepsilon}$, then $(e,\widetilde{\omega})\in \underline{B}^s_0$, and we consider $\lambda_0^s(e,\widetilde{\omega})\in \underline{E}^{\uparrow \mathbf{I}}$ and then we can define $\lambda_{\varepsilon}(e,\omega)=\lambda^s_0(e,\widetilde{\omega})\circ \theta^{-1}$.
For this we have $\lambda_{\varepsilon}(e,\omega)(0)=\lambda^s_0(e,\widetilde{\omega})(\theta^{-1})(0))= \lambda^s_0(e,\widetilde{\omega})(0)=e$ and
$p\circ \lambda_{\varepsilon}(e,\omega)=p\circ \lambda^s_0(e,\omega)\circ \theta^{-1}=\widetilde{\omega}\circ\theta^{-1}=\omega$.
We proceed similar for $\lambda^{\varepsilon}$.
\end{proof}
\begin{thm}\label{thm.5.4}
If $p:\underline{E}\rightarrow \underline{B}$ is a directed weak fibration with $\underline{E}\neq \emptyset$ and $\underline{B}$ is a directed path connected space, then $p$ is surjective and fibres of $p$ have all the same directed homotopy type.
\end{thm}
\begin{proof} Let $e$ be a point in $\underline{E}$  and $b\in \underline{B}$ an arbitrary point. If $p(e)\preceq b$ let $\omega\in d\underline{B}$ be a
d-path such that $\omega(0)=p(e)$ and $\omega(1)=b$. Moreover obviously we can suppose that $\omega$ is lower semistationary. Then by Cor.
\ref{cor.3.6}, there is $\omega'\in d\underline{E}$ with $\omega'(0)=e$ and $p\circ \omega'=\omega$. It follows that $b=\omega(1)=p(\omega'(1))$, i.e.,
$e\in p(|\underline{E}|)$. If $b\preceq p(e)$, and $\widetilde{\omega}\in d\widetilde{B}$ satisfies $\widetilde{\omega}(0)=b$, and for $t\in [0,1/2]$,
$\widetilde{\omega}(t)=\widetilde{\omega}(1)=p(e)$, then for $\widetilde{\omega}'\in d\underline{E}$ satisfying $\widetilde{\omega}'(1)=e $ and $p\circ
\widetilde{\omega}'=\widetilde{\omega}$, we have $b=\widetilde{\omega}(0)=p(\widetilde{\omega}'(0))$. Finally, if $b$ and $p(e)$ are joined by directed paths via the points
$b_1,...,b_n\in \underline{B}$, we deduce at first that $b_n\in p(\underline{E})$, then $b_{n-1}\in p(\underline{E)}$ and so on.

For the second part of the theorem, consider  $b_1,b_2\in \underline{B}$ and suppose at first that there is a d-path $\omega$ with ends $\omega(0)=b_1$,
$\omega(1)=b_2$. We associate to this path the path $\widetilde{\omega}\in d\underline{B}$ defined by

\[\widetilde{\omega}(t)=\left\{ \begin{array}{ll}
\omega(0)=b_1, & \mbox{if $0\leq t\leq 1/3 $},\\
\omega(3t-1), & \mbox{if $1/3\leq t\leq 2/3 $},\\\omega(1)=b_2,&\mbox{if $2/3\leq
t\leq 1$}.\end{array} \right.
\]
Now, by Corollary \ref{cor.5.3} we can consider the directed maps $\lambda_{1/3}:\underline{B}_{1/3}\rightarrow \underline{E}^{\uparrow \mathbf{I}}$ and
$\lambda^{2/3}:\underline{B}^{2/3}\rightarrow \underline{E}^{\uparrow \mathbf{I}} $. If $e\in \uparrow p^{-1}(b_1)$, $p(e)=b_1=\omega(0)=
\widetilde{\omega}(t), (\forall) t\in [0,1/3]$, such that $(e,\widetilde{\omega})\in \underline{B}_{1/3}$. And because $p\lambda_{1/3}(e,\widetilde{\omega})(1)=
\widetilde{\omega}(1)=\omega(1)=b_2$,$\Rightarrow $ $\lambda_{1/3}(e,\widetilde{\omega})(1)\in \uparrow p^{-1}(b_2)$. In this way we obtain a directed map
$$f:\uparrow p^{-1}(b_1)\rightarrow \uparrow p^{-1}(b_2), f(e)=\lambda_{1/3}(e,\widetilde{\omega})(1)$$
Analogously, we have a directed map
$$ g:\uparrow p^{-1}(b_2)\rightarrow \uparrow p^{-1}(b_1), g(e')=\lambda^{2/3}(e',\widetilde{\omega})(0)$$
based on the fact that $e'\in \uparrow p^{-1}(b_2)\Rightarrow p(e')=b_2=\widetilde{\omega}(t),(\forall) t\in [2/3,1]\Rightarrow (e',\widetilde{\omega})\in \underline{B}^{2/3}$,
and $p\lambda^{2/3}(e',\widetilde{\omega})(0)=\widetilde{\omega}(0)=b_1$.
Now
$$ (g\circ f)(e)=\lambda^{2/3}(\lambda_{1/3}(e,\widetilde{\omega})(1),\widetilde{\omega})(0).$$

Now for each $t\in [0,1]$ we consider  the following d-path:

\[\theta_t(t')=\left\{ \begin{array}{ll}
\widetilde{\omega}(\frac{3}{2}tt'), & \mbox{if $0\leq t'\leq \frac{2}{3} $},\\
\widetilde{\omega}(t), & \mbox{if $\frac{2}{3}\leq t'\leq 1 $}.\end{array} \right.
\]

Then we define

$$ \varphi:\uparrow p^{-1}(b_1)\times \uparrow \mathbf{I}\rightarrow \uparrow p^{-1}(b_1)$$
by
$$ \varphi(e,t)=\lambda^{2/3}(\lambda_{1/3}(e,\widetilde{\omega})(t),\theta_t)(0),$$

The map  is well defined because $(\lambda_{1/3}(e,\widetilde{\omega})(t),\theta_t)\in \underline{B}^{2/3}$.
For the directed homotopy $\varphi$, we have $\varphi(e,1)=(g\circ f)(e)$, and $\varphi(e,0)=\lambda^{2/3}(\lambda_{1/3}
(e,\widetilde{\omega})(0),\theta_0)(0)=\lambda^{2/3}(e,(\widetilde{\omega})_0)(0)$.
And, therefore $\underline{\varphi:\varphi_0\preceq_d g\circ f}$.
Now we define a new homotopy, $\psi:\uparrow p^{-1}(b_1)\times \uparrow \mathbf{I}\rightarrow \uparrow p^{-1}(b_1)$, by
$$\psi(e,t)=\lambda^{2/3}(e,(\widetilde{\omega})_t)(t),$$
for which $\psi(e,1)=\lambda^{2/3}(e,\widetilde{\omega})(1)=e$, and $\psi(e,0)=\lambda_{1/3}(e,(\widetilde{\omega})_0)(0)
=\varphi_0(e)$. Therefore $\underline{\psi:\varphi_0\preceq_d id_{\uparrow p^{-1}(b_1)}}$. From the above underlined relations we have that $g\circ f
\simeq_d id_{\uparrow p^{-1}(b_1)}$.
For $f\circ g$, we have
$$ (f\circ g)(e')=\lambda_{1/3}(\lambda^{2/3}(e'\omega)(0),\omega)(1),$$
and for this we consider the homotopy $\varphi':\uparrow p^{-1}(b_2)\times \mathbf{I} \rightarrow \uparrow p^{-1}(b_2)$, defined by
$$\varphi'(e',t)=\lambda_{1/3}(\lambda^{2/3}(e',\widetilde{\omega})(t),\theta'_t)(1),$$
where
\[\theta'_t(t')=\left\{ \begin{array}{ll}
\widetilde{\omega}(t), & \mbox{if $0\leq t'\leq \frac{1}{3} $},\\
\widetilde{\omega}(\frac{(3t'-1)t}{2}), & \mbox{if $\frac{1}{3}\leq t'\leq 1 $}.\end{array} \right.
\]

For this we have $\varphi'(e',0)=(f\circ g)(e')$, and $\varphi'(e',1)=\lambda_{1/3}(\lambda^{2/3}(e',\widetilde{\omega})(1),(\widetilde{\omega})_1)(1)=
\lambda_{1/3}(e',\theta'_1)(1)$. Therefore $\underline{\varphi':f\circ g\preceq_d \varphi'_1}$. Then we define $\psi':\uparrow p^{-1}(b_2)\times
\mathbf{I}\rightarrow \uparrow p^{-1}(b_2)$ by
$$ \psi'(e',t)=\lambda_{1/3}(e',(\widetilde{\omega})_t)(t).$$
For this we have $\psi'(e',0)=e'$ and $\psi'(e',1)=\varphi'_1$.
Therefore $\underline{id_{\uparrow p^{-1}(b_2)}\preceq_d \varphi'_1}$. From the last underlined relations we have that
$f\circ g\simeq_d id_{\uparrow p^{-1}(b_2)}$. Therefore, if $b_1\preceq_d b_2$ (or $b_2\preceq_d b_1)$, we have proved that
$\uparrow p^{-1}(b_1)\simeq_d \uparrow p^{-1}(b_2)$.

If $b_1$ and $b_2$ can't be joined by a directed path, but if for example $b_1\preceq b_3$ and $b_2\preceq b_3$, then as proved above
 $\uparrow p^{-1}(b_1)\simeq_d \uparrow p^{-1}(b_3)$ and $\uparrow p^{-1}(b_2)\simeq_d \uparrow p^{-1}(b_3)$, which again implies
 $\uparrow p^{-1}(b_1)\simeq_d \uparrow p^{-1}(b_2)$. And the general case follows in the same way.

\end{proof}
\begin{ex}\label{ex.5.5}If $p:\underline{E}\rightarrow \underline{B}$ is a directed weak fibration with $(\lambda^s_\alpha)_{\alpha=0,1}$ a directed semistationary lifting pair, and if $f:\underline{B}'\rightarrow \underline{B}$ is a directed map, then for the directed weak fibration $p_f:\underline{E}_f\rightarrow \underline{B}'$ (Prop. \ref{prop.3.14} )we have the following:\\\\\
$\underline{B}'^s_0=\{((b',e),\omega')\in \underline{E}_f\times \underline{B}'^{\uparrow \mathbf{I}}|\omega'(t)=b',(\forall)t\in [0,\frac{1}{2}]\}$, $\underline{B}'^s_1=\{((b',e),\omega')\in \underline{E}_f\times \underline{B}'^{\uparrow \mathbf{I}}|\omega'(t)=b',(\forall)t\in [\frac{1}{2},1]\}$. For these d-spaces we can define ${\lambda'}^s_\alpha:\underline{B}'^s_\alpha\rightarrow \underline{E}_f^{\uparrow \mathbf{I}}$ , $\alpha=0,1$, by $\lambda'^s_\alpha((b',e),\omega'))(t)=(\omega'(t),\lambda^s_\alpha(e,f\circ \omega')(t))$. For this we have $\lambda'^s_\alpha((b',e),\omega'))(\alpha)=(\omega'(\alpha),\lambda^s_\alpha(e,f\circ \omega')(\alpha))=(b',e)$ and $p_f((\lambda'^s_\alpha((b',e),\omega'))(t))=\omega'(t)$. Therefore, $(\lambda'^s_\alpha)_{\alpha=0,1}$ is a directed semistationary lifting pair for $p_f$.
\end{ex}
\begin{ex}\label{ex.5.6}Let $p:\underline{E}\rightarrow \underline{B}$ be directed weak fibration with $(\lambda^s_\alpha)_{\alpha=0,1}$ a directed semistationary lifting pair. Consider the opposite directed weak fibration $p:\underline{E}^{op}\rightarrow \underline{B}^{op}$ (Corollary \ref{cor.3.13}). For this map we have:$(\underline{B}^{op})^s_0=\{(e,\omega)\in \underline{E}^{op}\times (\underline{B}^{op})^{\uparrow \mathbf{I}}|\omega(t)=p(e),(\forall) t\in [0,\frac{1}{2}]\}$  and  $(\underline{B}^{op})^s_1=\{(e,\omega)\in \underline{E}^{op}\times (\underline{B}^{op})^{\uparrow \mathbf{I}}|\omega(t)=p(e),(\forall) t\in [\frac{1}{2},1]\}$.

Now it's easy to see that if $(e,\omega)\in (B^{op})^s_\alpha$, then $(e,\omega^{op})\in \underline{B}^s_{1-\alpha}$, such that we can define $(\lambda^{op})^s_\alpha:(\underline{B}^{op})^s_\alpha\rightarrow (\underline{E}^{op})^{\uparrow \mathbf{I}}$ by

$$ (\lambda^{op})^s_\alpha(e,\omega)=(\lambda^s_{1-\alpha}(e,\omega^{op}))^{op}$$

For these d-maps we have:$ (\lambda^{op})^s_\alpha(e,\omega)(\alpha)=\lambda^s_{1-\alpha}(e,\omega{op})(1-\alpha)=e$,\\ and
$p(\lambda^{op})^s_\alpha(e,\omega)(t)=(p\lambda^s_{1-\alpha}(e,\omega^{op})(1-t)=\omega^{op}(1-t)=\omega(t)$.

Therefore $((\lambda^{op})^s_\alpha))_{\alpha=0,1}$ is a directed semistationary lifting pair for $$p^{op}:=p:\underline{E}^{op}\rightarrow \underline{B}^{op}.$$
\end{ex}

\begin{rem}\label{rem.5.7}A theorem similar to Theorem \ref{thm.5.2} exists also in the undirected case. But in that case it is sufficient to have a  lifting function for the stationary path on the interval $[0,\frac{1}{2}]$ since there the spaces $B^s_0$ and $B^s_1$ are homeomorphic by the correspondence $(e,\omega)\in B^s_0\rightarrow (e,\omega^{op})\in B^s_1$. And if $\lambda^s_0$ exists, then $\lambda^s_1$ is defined by $(\lambda^s_1(e,\omega)=(\lambda^s_0(e,\omega^{op}))^{op}$. In the general directed case, for an arbitrary directed map $p:\underline{E}\rightarrow \underline{B}$, the spaces $\underline{B}^s_0$ and $\underline{B}^s_1$ are independent. But if $p$ is a directed weak fibration, then we have the following theorem.
\end{rem}
\begin{thm}\label{thm.5.8}If $p:\underline{E} \rightarrow \underline{B}$ is a directed weak fibration, then the d-spaces $\underline{B}^s_0$ and $\underline{B}^s_1$ are d-homotopy equivalent.
\end{thm}
\begin{proof}By hypothesis $p$ admits a directed semistationary lifting pair $(\lambda^s_\alpha)_{\alpha=0,1}$, which we use in the proof. If $\omega\in d\underline{B}$ is lower semistationary, $\omega(t)=\omega(\frac{1}{2}),(\forall)t\in [0,\frac{1}{2}]$, then we define $\omega_{+}\in d\underline{B}$ by

\[\omega_{+}(t)=\left\{ \begin{array}{ll}
\omega(2t), & \mbox{if $0\leq t\leq \frac{1}{2}$},\\
\omega(1), & \mbox{if $\frac{1}{2}\leq t\leq 1 $}.\end{array} \right.
\]
This path is upper semistationary, and if $(e,\omega)\in \underline{B}^s_0)$, we have $\lambda^s_0(e,\omega)(1)\in \underline{E}$ and $p(\lambda^s_0(e,\omega)(1))=\omega(1)=\omega_{+}(1)$, so that we can define the directed map $f:\underline{B}^s_0\rightarrow \underline{B}^s_1$ by
$$ f(e,\omega)=(\lambda^s_0(e,\omega)(1),\omega_{+}).$$
Similarly we define $g:\underline{B}^s_1 \rightarrow B^s_0$ by
$$ g(e',\omega')=(\lambda^s_1(e',\omega')(0),\omega'_{-}),$$
where

\[\omega'_{-}(t)=\left\{ \begin{array}{ll}
\omega'(0), & \mbox{if $0\leq t\leq \frac{1}{2}$},\\
\omega'(2t-1), & \mbox{if $\frac{1}{2}\leq t\leq 1 $}.\end{array} \right.
\]
It follows that for $g\circ f:\underline{B}^s_0\rightarrow \underline{B}^s_0$ we have
$$(g\circ f)(e,\omega)=(\lambda^s_1(\lambda^s_0(e,\omega)(1),\omega_{+})(0),(\omega_{+})_{-}).$$

Now if $(e,\omega)\in \underline{B}^s_0 $ , then $\omega_t$, defined as above by $\omega_t(t')=\omega(tt')$ is lower semistationary and we can define
$\varphi:\underline{B}^s_0\times \uparrow \mathbf{I}\rightarrow \underline{B}^s_0$, by
$$\varphi((e,\omega),t)=(\lambda^s_1(\lambda^s_0(e,\omega)(t),(\omega_t)_{+})(0),((\omega_t)_{+})_{-}).$$
For this d-homotopy we have $\varphi((e,\omega),1)=(g\circ f)(e,\omega)$, and $\varphi((e,\omega),0)=(\lambda^s_1(e,\omega_0)(0),\omega_0)$. Hence
$\underline{\varphi:\varphi_0\preceq_d g\circ }f$. Then we define a new d-homotopy $\psi:\underline{B}^s_0\times \uparrow\mathbf{I}\rightarrow \underline{B}^s_0$ by
$$\psi((e,\omega),t)=(\lambda^s_1(e,\omega_0)(t),\omega_t).$$
This is well defined since for $t'\in [0,1]$ and $t\in [0,1], tt'\in [0,1/2]$ such that $\omega_t(t')=\omega(tt')=p(e)=p\lambda^s_1(e,\omega_0)(t)=\omega_0(t)$, and therefore $(\lambda^s_1(e,\omega_0)(t),\omega_t)\in \underline{B}^s_0$.

This induces $\underline{\psi:\varphi_0\preceq_d id_{\underline{B}^s_0}}$. From the last two underlined relations, we deduce that
$g\circ f\simeq_d id_{\underline{B}^s_0}$.

For $f\circ g$, we have
$$ (f\circ g)(e',\omega')=(\lambda^s_0(\lambda^s_1(e',\omega')(0),\omega'_{-})(1),((\omega'_{-})_{+}),$$
and then define $\varphi':\underline{B}^s_1\times \uparrow \mathbf{I}\rightarrow \underline{B}^s_1$ by
$$\varphi'((e',\omega'),t)=(\lambda^s_0(\lambda^s_1(e',\omega')(t),(\omega'_t)_{-})(1),((\omega'_t)_{-})_{+}),$$
for which we have $\underline{\varphi': f\circ g\preceq_d\varphi'_1}$. And further by the d-homotopy $\psi':\underline{B}^s_1\times \uparrow \mathbf{I}\rightarrow \underline{B}^s_1$, defined by
$$\psi'((e',\omega'),t)=(\lambda^s_0(e',\omega')(t),\omega'_t),$$
we have $\psi':\underline{id_{\underline{B}^s_1}\preceq_d \varphi'_1}$. The two last  underlined relations imply
$f\circ g \simeq_d id_{\underline{B}^s_1}$ . And by this the proof is completed.
\end{proof}
\begin{cor}\label{cor.5.9} If $p:\underline{E}\rightarrow \underline{B}$ is a directed weak fibration and $\varepsilon\in (0,1)$, then the d-spaces
$\underline{B}_{\varepsilon}$ and $\underline{B}^{\varepsilon}$ are d-homotopy equivalent.
\end{cor}

\section{Directed fiber homotopy equivalence}
\begin{defn}\label{def.6.1}Let $p:\underline{E}\rightarrow \underline{B}$, $p':\underline{E}'\rightarrow \underline{B}$ and $f:E\rightarrow E'$ directed maps with $p'\circ f=p$ ($f$ is a morphism from $p$ to $p'$). We say that $f$ is a directed fibre homotopy equivalence if there exists $g:\underline{E}'\rightarrow \underline{E}$ a morphism from $p'$ to $p$, such that $g\circ f \stackunder{(p)}{\simeq_d}id_{\underline{E}}$ and $f\circ g\stackunder{(p')}{\simeq_d}id_{\underline{E}'}$.
\end{defn}
In usual homotopy theory, the WCHP for the maps $p$ and $p'$ ensure that $f$ is a directed fiber homotopy equivalence if it is a simple directed homotopy equivalence, \cite{Dold1}. The proof is not simple even in that case where some homotopies at the same time as their reverses are used.
This is not possible in the case of directed homotopy, but keeping the idea and modifying the calculations we succeeded to prove the following theorem.

\begin{thm}\label{thm.6.2}Let $p:\underline{E}\rightarrow \underline{B}$ , $p':\underline{E}'\rightarrow \underline{B}$ be directed weak fibrations.
Then a directed map $f:\underline{E}\rightarrow \underline{E}'$ over $\underline{B}$, $p'\circ f=p$, is a directed fibre homotopy equivalence if and only
if it is an ordinary directed homotopy equivalence.
\end{thm}
The proof of this theorem is based on the following two lemmas.
\begin{lem}\label{lem.6.3}Suppose given a commutative diagram
\[\xymatrix{\underline{E}\ar[rr]^f \ar[dr]_p && \underline{E}'\ar[ld]^{p'}\\& \underline{B} &}\]
$p'\circ f=p$, and a d-map $f':\underline{E}'\rightarrow \underline{E}$,  such that $f\circ f'\simeq_d id_{\underline{E}'}$. Then, if $p$ is a directed weak fibration, there exists a d-map $\widetilde{f'}:\underline{E}'\rightarrow \underline{E}$ over $\underline{B}$, $p\circ \widetilde{f'}=p'$, such that
$f\circ \widetilde{f'} \simeq_d id_{\underline{E}'}$.
\end{lem}
\begin{proof}Suppose that $f\circ f'\preceq_d id_{\underline{E}'}$ by the directed homotopy $\varphi:\underline{E}'\times \uparrow \mathbf{I}
\rightarrow \underline{E}'$ , with $\varphi_0=f\circ f',\varphi_1=id_{\underline{E}'}$. Obviously we can suppose that $\varphi$ is lower semistationary,
$\varphi(e',t)=f(f'(e'))$ for all $e'\in \underline{E}'$ and $t\in [0,\frac{1}{2}]$. Consider $p'\circ \varphi:\underline{E}'\times \uparrow \mathbf{I}
\rightarrow \underline{B}$, which is also lower semistationary and which satisfies $(p'\circ\varphi)(e',t)=p'(f\circ f')(e')=(p\circ f')(e'),
(\forall)e'\in \underline{E}', (\forall)t\in [0,\frac{1}{2}]$, i.e., $p\circ f'=\varphi\circ\partial^0$,

$$
\xymatrix{
\underline{E} \ar[rr]^p & & \underline{B}\\
\underline{E'} \ar[u]^{f'} \ar[rr]_{\partial^0} & &
\underline{E'}\times \uparrow
\textbf{I}\ar[u]_{p'\circ \varphi} }
$$
By hypothesis (cf. Theorem \ref{thm.4.2}) there exists $\varphi':\underline{E}'\times \uparrow \mathbf{I}\rightarrow \underline{E}$, with $\varphi'\circ
\partial^0=f'$ and $p\circ \varphi'=p'\circ \varphi$. Define $\widetilde{f'}:\underline{E}'\rightarrow \underline{E}$, by $\widetilde{f'}=\varphi'\circ
\partial^1$. For this we have $p\circ \widetilde{f'} =p\circ \varphi'\circ\partial^1=p'\circ \varphi\circ \partial^1=p'\circ id_{\underline{E}'}=p'$, and
$\varphi':f'\preceq_d \widetilde{f'}$. Then $f\circ \varphi':f\circ f'\preceq_d f\circ \widetilde{f'}$, which together  with $f\circ f'\simeq_d id_{\underline{E}'}$ implies $f\circ \widetilde{f'} \simeq_d id_{\underline{E}'}$.

If $id_{\underline{E}'}\preceq_d f\circ f'$ by $\overline{\varphi }$, with $\overline{\varphi}\circ \partial^0=id_{\underline{E}'},
\overline{\varphi}\circ\partial^1=f\circ f'$, and $\overline{\varphi}$ an upper semistationary directed homotopy, we apply Theorem \ref{thm.4.2} for the
commutative diagram
$$
\xymatrix{
\underline{E} \ar[rr]^p & & \underline{B}\\
\underline{E'} \ar[u]^{f'} \ar[rr]_{\partial^1} & &
\underline{E'}\times \uparrow
\textbf{I}\ar[u]_{p'\circ \overline{\varphi}} }
$$
$(p'\circ\overline{\varphi})\circ \partial^1= p'\circ f\circ f'=p\circ f'$, and let $\overline{\varphi}':\underline{E}'\times \uparrow \mathbf{I}\rightarrow
\underline{E}$ be satisfying $\overline{\varphi}'\circ\partial^1=f'$ and $p\circ \overline{\varphi}'=p'\circ \overline{\varphi}$. Define $\widetilde{f'}=
\overline{\varphi}'\circ \partial^0$. For this we have $p\circ \widetilde{f'}=p\circ \overline{\varphi}'\circ\partial^0=p'\circ \overline{\varphi}\partial^0=p'$ and $\overline{\varphi}':\widetilde{f'}\preceq_d f'$ $\Rightarrow $ $f\circ \overline{\varphi}':f\circ \widetilde{f'}
\preceq_d f\circ f'\simeq_d id_{\underline{E}'}$.

\end{proof}
\begin{lem}\label{lem.6.4}Let $p:\underline{E}\rightarrow  \underline{B}$ be a directed weak fibration and $g: \underline{E}\rightarrow \underline{E}$
be a fibrewise d-map, $p\circ g=p$, such that  $g\simeq_d id_{\underline{E}}$. Then there exists $g':E\rightarrow E$, a fibrewise d-map, $p\circ g'
=p$, such that $g\circ g'\stackunder{(p)}{\simeq_d}id_{\underline{E}}$.
\end{lem}
\begin{proof}Suppose that $g\preceq_d id_{\underline{E}}$ by  an upper semistationary d-homotopy $\varphi:\underline{E}\times \uparrow \mathbf{I}\rightarrow
\underline{E}$, i.e., $\varphi\circ\partial^0=g$ and $\varphi(e,t)=e,(\forall)e\in \underline{E},(\forall) t\in [\frac{1}{2},1]$. Then we have the following commutative diagram
$$
\xymatrix{
\underline{E} \ar[rr]^p & & \underline{B}\\
\underline{E} \ar[u]^{id_{\underline{E}}} \ar[rr]_{\partial^1} & &
\underline{E}\times \uparrow
\textbf{I}\ar[u]_{p\circ \varphi} }
$$
By hypothesis (cf. Theorem \ref{thm.4.2}), there exists $\psi:\underline{E}\times \uparrow \mathbf{I}\rightarrow \underline{E}$, with $\psi\circ
\partial^1=id_{\underline{E}}$ and $p\circ \psi=p\circ \phi$. Define $g':\underline{E}\rightarrow \underline{E}$, by
$$ g'=\psi\circ \partial^0.$$

 For this d-map we have: $\underline{p\circ g'}= p\circ\psi\circ\partial^0=p\circ\varphi\circ\partial^0=p\circ g\underline{=p}$. And we can prove that
$g\circ g'\stackunder{(p)}{\simeq_d}id_{\underline{E}}$.

Define $F:\underline{E}\times \uparrow \mathbf{I}\rightarrow \underline{E}$, by
\[F(e,s)=\left\{ \begin{array}{ll}
g\psi(e,2s), & \mbox{if $0\leq s\leq 1/2 $},\\
\varphi(e,2s-1), & \mbox{if $1/2\leq s\leq 1 $}.\end{array} \right.
\]
For $s=\frac{1}{2}$ we have $g\psi(e,1)=g(e)=\varphi(e,0)$, so that $F$ is well defined and a directed homotopy. Then $F(e,0)=g\psi(e,0)=gg'(e)=
(g\circ g')(e)$, and $F(e,1)=\varphi(e,1)=e$,$(\forall)e\in \underline{E}$. Therefore
$$F:g\circ g'\preceq_d id_{\underline{E}}.$$
In addition,
\[pF(e,s)=\left\{ \begin{array}{ll}
pg\psi(e,2s)=p\psi(e,2s)=p\varphi(e,2s) & \mbox{if $0\leq s\leq 1/2 $},\\
p\psi(e,2s-1)=p\varphi(e,2s-1), & \mbox{if $1/2\leq s\leq 1 $}.\end{array} \right.
\]

Now we define $\Phi:\underline{E}\times \uparrow \mathbf{I}\times \uparrow \mathbf{I}\rightarrow \underline{B}$, by
\[\Phi(e,s,t)=\left\{ \begin{array}{ll}
(p\circ F)(e,2st) & \mbox{if $0\leq t\leq \frac{1}{2} $},\\
pF(e,s), & \mbox{if $\frac{1}{2}\leq t\leq 1 $}.\end{array} \right.
\]
With this from the commutative diagram
$$
\xymatrix{
\underline{E} \ar[rr]^p & & \underline{B}\\
\underline{E}\times \uparrow \mathbf{I} \ar[u]^{F} \ar[rr]_{\partial^1} & &
\underline{E}\times \uparrow \mathbf{I}\times \uparrow
\textbf{I}\ar[u]_\Phi }
$$
there exists $\widetilde{\Phi}: \underline{E}\times \uparrow \mathbf{I}\times \uparrow \mathbf{I}\rightarrow \underline{E}$ satisfying
$$\widetilde{\Phi}(e,s,1)=F(e,s),p\circ \widetilde{\Phi}=\Phi.$$
Now we consider the following d-maps
$$\widetilde{\Phi}_{(s,t)}:\underline{E}\rightarrow \underline{E},\widetilde{\Phi}_{(s,t)}(e)=\widetilde{\Phi}(e,s,t),e\in \underline{E}, s,t\in [0,1].$$
Now by the above relations we have the following directed fiber homotopy equivalences:
$$\widetilde{\Phi}_{(0,0)}\stackunder{(p)}{\simeq_d}\widetilde{\Phi}_{(1,0)},\widetilde{\Phi}_{(0,0)}\stackunder{(p)}{\simeq_d}\widetilde{\Phi}_{(0,1)},
\widetilde{\Phi}_{(1,0)}\stackunder{(p)}{\simeq_d}\widetilde{\Phi}_{(1,1)},$$
since $p\widetilde{\Phi}(e,s,0)=\Phi(e,s,0)=pF(e,0)=p(g\circ g')(e)=p(e), p\widetilde{\Phi}(e,0,t)=\Phi(e,0,t)=pF(e,0)=p(e)$, and $ p\widetilde{\Phi(}e,1,t)=pF(e,1)=p\varphi(e,1)=p(e)$. Thus
$$\widetilde{\Phi}_{(0,1)}\stackunder{(p)}{\simeq_d}\widetilde{\Phi}_{(1,1)}.$$
And because $g\circ g'=F_0=\widetilde{\Phi}_{(0,1)} $ and $id_{\underline{E}}=F_1=\widetilde{\Phi}_{(1,1)}$, we obtain
$g\circ g'\stackunder{(p)}{\simeq_d}id_{\underline{E}}$.

If $id_{\underline{E}}\preceq_d g$, we start instead of $\varphi$ with a d-homotopy $\overline{\varphi}:\underline{E}\times
\uparrow \mathbf{I}\rightarrow \underline{E}$ satisfying $\overline{\varphi}(e,t)=e,(\forall)e\in E,(\forall)t\in [0,\frac{1}{2}]$ and $\overline{\varphi}\circ
\partial^1=g$ and then follow the above scheme.

Now consider the case $g\preceq_d h  _d\succeq id_{E\underline{}}$. From the proof of Lemma \ref{lem.6.3} for $p'=p,f=id_{\underline{E}}$ and $f'=h$,
there exists a d-map $\widetilde{h}:\underline{E}\rightarrow \underline{E}$ over $\underline{B}$, $p\circ \widetilde{h}=p$, and such that
$$ g\preceq_d h\preceq _d \widetilde{h}  _d\succeq id_{\underline{E}}.$$
Then there exists $\widetilde{h}':\underline{E}\rightarrow \underline{E}$, with $p\widetilde{h}'=p$ and $\widetilde{h}\circ \widetilde{h}'\stackunder{(p)}{\simeq_d} id_{\underline{E}}$. Further, repeating the scheme for the case $g\preceq_d id_{\underline{E}}$
with $\widetilde{h}$ instead of $id_{\underline{E}}$, there exists $g_1:\underline{E}\rightarrow \underline{E}$ satisfying $pg_1=p$ and $g\circ g_1\stackunder{(p)}{\simeq_d}\widetilde{h}$. Then $(g\circ g_1)\circ \widetilde{h}'\stackunder{(p)}{\simeq_d}\widetilde{h}\circ \widetilde{h}'\Rightarrow g\circ (g_1\circ \widetilde{h}')\stackunder{(p)}{\simeq_d}id_{\underline{E}}$.

The conclusion for the general case $g\preceq_d ..._d\succeq id_{\underline{E}}$ follows in a similar way.

\end{proof}

\textit{Proof of Theorem \ref{thm.6.2}}. We prove the "if-part". Assume then that $f:\underline{E}\rightarrow \underline{E}'$ is an (ordinary) directed homotopy equivalence; let $f':\underline{E}'\rightarrow \underline{E}$ be a directed homotopy inverse, $f\circ f'\simeq_d id_{\underline{E}'}$. Then by Lemma \ref{lem.6.3}, there is $\widetilde{f'}:\underline{E}'\rightarrow \underline{E}$ over $\underline{B}$,$p\circ \widetilde{f'}=p'$, such that
$f\circ \widetilde{f'}\simeq_d id_{E'}$. Now for the composition $g:=f\circ \widetilde{f'}:\underline{E}'\rightarrow \underline{E}'$ we have
$p'\circ g=p'\circ f\circ \widetilde{f'}=p\circ \widetilde{f'}=p'$ , and $g\simeq_d id_{\underline{E}'}$. Then by Lemma \ref{lem.6.4}, for the directed
weak fibration $p'$ and the map $g$, there exists $g':E'\rightarrow E'$, over $\underline{B}$, $p'\circ g'=p'$, such that $ g\circ g'\stackunder{(p)}
{\simeq_d}id_{\underline{E}'} \Rightarrow f\circ (\widetilde{f'}\circ g') \stackunder{(p)}{\simeq_d} id_{\underline{E}'}$. Therefore for
$\overline{f'}:=\widetilde{f'}\circ g':\underline{E}'\rightarrow \underline{E}$, we have
$$f\circ \overline{f'} \stackunder{(p)}{\simeq_d} id_{\underline{E}'}.$$
On the other hand, because we have also $f'\circ f\simeq_d id_{\underline{E}}$, which implies $f'\circ f\circ \overline{f'} \simeq_d f'$ , it follows
$\overline{f'}\simeq_d f'$, i.e., $\overline{f'}$ is directed homotopy equivalence, over $\underline{B}$. Then as above, using the property of a directed
weak fibration for $p$, there exists for $\overline{f'}$ a d-map $\widehat{f}:\underline{E}\rightarrow \underline{E}'$ over $B$, such that
$\overline{f'}\circ \widehat{f}\stackunder{(p)}{\simeq_d} id_{\underline{E}}$. Then $f\circ \overline{f'}\circ \widehat{f}\stackunder{(p)}{\simeq_d} f$, so that $\widehat{f}\stackunder{(p)}{\simeq_d} f$, and finally
$$\overline{f'}\circ f\stackunder{(p)}{\simeq_d}id_{\underline{E}}.$$
This completes the proof.

\begin{defn}\label{def.6.5} A d-map $p:\underline{E}\rightarrow \underline{B}$ is called \emph{d-shrinkable} if one of the following equivalent properties is satisfied:

(a) $p$ is a directed fibre homotopy equivalence (viewed as a d-map over $\underline{B}$ into $id_{\underline{B}}$),

(b) $p$ is directed dominated by $id_{\underline{B}}$ ,

(c) there is a d-section $s:\underline{B}\rightarrow \underline{E}$ , $p\circ s=id_{\underline{B}}$, and a vertical homotopy d-equivalence
$s\circ p\simeq_d id_{\underline{E}}.$

\end{defn}

By Theorem \ref{thm.6.2} we obtain:

\begin{cor}\label{cor.6.6} If $p:\underline{E}\rightarrow \underline{B}$ is directed weak fibration, then $p$ is d-shrinkable if and only if $p$ is
a directed homotopy equivalence.
\end{cor}

\begin{lem}\label{lem.6.7} If $p:\underline{E}\rightarrow \underline{B}\times \uparrow \mathbf{I}$ has dWHCP, then there exists a d-map $R:\underline{E}\times
\uparrow \mathbf{I}\rightarrow \underline{E}$ such that

(i)  \qquad $pR(e,t)=(\pi(e),t)$,

(ii) \qquad $r\stackunder{(p)}{\simeq_d}id_{\underline{E}}$,

where $\pi:\underline{E}\rightarrow \underline{B}$, $\rho:\underline{E}\rightarrow \underline{E}$ are defined by $p(e)=(\pi(e),\rho(e)$, $r(e)=R(e,\rho(e))$,
$e\in \underline{E}$.
\end{lem}
\begin{proof} Define $\varphi:(\underline{E}\times \uparrow \mathbf{I})\times \uparrow \mathbf{I}\rightarrow \underline{B}\times \uparrow \mathbf{I}$ by
\[\varphi(e,t_1,t_2))=\left\{ \begin{array}{ll}
(\pi(e),\rho(e))=p(e), & \mbox{if $0\leq t_2\leq \frac{1}{2}, t_1\in [0,1] $},\\
(\pi(e),2(1-t_2)\rho(e)+t_1(2t_2-1)), & \mbox{if $\frac{1}{2}\leq t_2\leq 1, \rho(e)\leq t_1 $},\\(\pi(e),2t_1(1-t_2)+(2t_2-1)\rho(e)),&\mbox{if $\frac{1}{2}\leq
t_2\leq 1,\rho(e)\geq t_1$}.\end{array} \right.
\]
This is a well defined d-map and it is a lower semistationary d-homotopy with $\varphi\circ\partial^0=p(e)=(p\circ pr_1)(e,t)$, for $pr_1:\underline{E}\times \uparrow \mathbf{I}$ is the projection $pr_1(e,t)=e$. Then, by hypothesis, there exists $\varphi':
(\underline{E}\times \uparrow \mathbf{I})\times \uparrow \mathbf{I}\rightarrow \underline{E}$ , with $p\circ \varphi'=\varphi$ and $\varphi'(e,t_1,0)=pr_1(e,t_1)=e, (\forall )t_1\in [0,1]$. Define $R:\underline{E}\times \uparrow \mathbf{I}\rightarrow \underline{E}$ by $R(e,t)=\varphi'(e,t,1)$.
This satisfies (1), since for $t_2=1$, and $t_1=t$, we have $pR(e,t)=p\varphi'(e,t,1))= \varphi(e,t,1)=(\pi(e),t)$. And (ii) follows from $r(e)=R(e,\rho(e))=\varphi'(e,\rho(e),1)$, i.e. $r=\varphi'(., \rho(.), 1)\stackunder{(p)}{\succeq_d}\varphi'(.,\rho(.),0)$ by $\kappa(e,t)=\varphi'(e,\rho(e),t)$ , with $p\kappa(e,t)=\varphi(e,\rho(e),t)=(\pi(e),t)=p(e)$.
\end{proof}
\begin{cor}\label{cor.6.8}With the same conditions and notations as in Lemma \ref{lem.6.7}, let $p^t:\underline{E}^t\rightarrow \underline{B}$ the part of $p:\underline{E}\rightarrow \underline{B}\times\uparrow \mathbf{I}$ over $\underline{B}\times \{t\}\approx \underline{B}$. Then the d-maps
$$ h^1: \underline{E}^0\rightarrow \underline{E}^1, h^1(x)=R(x,1),$$
$$ h^0: \underline{E}^1\rightarrow \underline{E}^0, h^0(y)=R(y,0)$$
are reciprocal directed fibre homotopy equivalences.
\end{cor}
\begin{proof}At first, $p^1h^1(x)=p^1R(x,1)=pR(x_1,1)=(\pi(x),1)\equiv \pi(x)\equiv (\pi(x),0)=p^0(x)$, therefore $h^1$ is a d-map over $\underline{B}$.
Analogously we have $p^0\circ h^0=p^1$. Then, $h^1\circ h^0=R(R((.,0),1)\stackunder{(p)}{\preceq_d}R(R(.,1),1)=r\circ r\stackunder{(p)}{\preceq_d}id_{\underline{E}^0}$. The first $\stackunder{(p)}{\preceq_d}$ is achieved by the (ordinary) d-homotopy $\varphi(y,\tau)=
R(R(y,\tau),1)$ and the second by Lemma \ref{lem.6.7} (ii). Therefore $h^1\circ h^0\simeq _d id_{\underline{E}^0}$. Similarly $h^0\circ h^1\simeq_d
id_{\underline{E}^1}$. It follows that $h^0$ is a d-map over $\underline{B}$ and a d-homotopy equivalence. Now since obviously $p^i, i=0,1,$ are
directed weak fibrations, by Theorem \ref{thm.6.2} we have that $h^0$ is a directed fibre homotopy equivalence with inverse $h^1$.
\end{proof}
Using the last two results we can prove the following corollary by paraphrasing a part of the proof of Theorem 6.3 in \cite{Dold1}, p.246.
\begin{cor}\label{cor.6.9}  Let $p:\underline{E}\rightarrow \underline{B}$, $p':\underline{E}'\rightarrow \underline{B}$ be directed weak fibrations
and $f:\underline{E}\rightarrow \underline{E}'$ a d-map over $B$, $p'\circ f=p$. Suppose that $\underline{B}$ is directed contractible to a point
$b\in \underline{B}$ (i.e., $id_{\underline{B}}\simeq_d c_b$, where $c_b$ is the constant map $c_b(\underline{B})=b$), and that the restriction \qquad
$f_b:\uparrow p^{-1}(b)\rightarrow \uparrow p'^{-1}(b)$  is a (ordinary) directed homotopy equivalence. Then $f$ is a directed fibre homotopy equivalence.
\end{cor}

Faculty of Mathematics

"Al. I. Cuza" University,

 700505-Iasi, Romania.

E-mail:ioanpop@uaic.ro

\end{document}